\documentclass[a4paper,12pt]{article}
\usepackage{amsmath,amsthm,amssymb,amscd}
\usepackage{ascmac}
\usepackage[arrow,matrix]{xy}
\usepackage{enumerate}
\usepackage{graphicx}

\theoremstyle{definition}
\numberwithin{equation}{section}
\newtheorem{thm}{Theorem}[section]
\newtheorem{dfn}[thm]{Definition}
\newtheorem{exa}[thm]{Example}
\newtheorem{prop}[thm]{Proposition}
\newtheorem{cor}[thm]{Corollary}
\newtheorem{lem}[thm]{Lemma}
\newtheorem{conj}[thm]{Conjecture}

\newtheorem{rem}[thm]{Remark}

\def\dim{\mathop{\mathrm{dim}}\nolimits}

\def\Hom{\mathop{\mathrm{Hom}}\nolimits}

\def\Lam{X_{\ast}(T)}
\def\lam{\check{\lambda}}
\def\la{\lambda}
\def\al{\check{\alpha}}
\def\ep{\varepsilon }

\newcommand{\mf}[1]{{\mathfrak{#1}}}
\newcommand{\mb}[1]{{\mathbf{#1}}}
\newcommand{\bb}[1]{{\mathbb{#1}}}
\newcommand{\mca}[1]{{\mathcal{#1}}}
\newcommand{\mr}[1]{{\mathrm{#1}}}

\title{Affine Springer fibers of type $A$ and combinatorics of diagonal coinvariants}
\author{Tatsuyuki Hikita\footnote{thikita@math.kyoto-u.ac.jp}}
\date{}
\begin{document}

\maketitle

\begin{abstract}
We calculate the Borel-Moore homology of affine Springer fibers of type $A$ associated to some regular semisimple nil elliptic elements. As a result, we obtain bigraded $\mf{S}_{n}$-modules whose bigraded Frobenius series are generalization of the symmetric functions introduced by Haglund, Haiman, Loehr, Remmel, and Ulyanov. 
\end{abstract}

\section{Introduction}\mbox{}

Let $\mu=(\mu_{1},\ldots,\mu_{l})$ be a partition of $n$. Let $x_{\mu}$ be a nilpotent element in the Lie algebra $\mf{sl}_{n}(\bb{C})$ whose Jordan blocks have sizes $\mu_{1},\ldots\mu_{l}$. Let $\mca{B}_{\mu}$ be the variety of Borel subalgebras of $\mf{sl}_{n}(\bb{C})$ containing $x_{\mu}$. This is known as the Springer fiber of type $A$ associated to $x_{\mu}$. The symmetric group $\mf{S}_{n}$ acts on the cohomology ring $$R_{\mu}=H^{\ast}(\mca{B}_{\mu}).$$ This action preserves the grading on $R_{\mu}$ given by cohomological degree. Hence one can consider its Frobenius series $\mca{F}(R_{\mu},z;q)$. The coefficient of $s_{\la}(z)$ in the Schur function expansion of $\mca{F}(R_{\mu},z;q)$ essentially coincides with the Kostka-Foulkes polynomial $K_{\la\mu}(q)$. This gives a geometric interpretation of Kostka-Foulkes polynomials. 

According to De Concini-Procesi and Tanisaki (\cite{MR629470}, \cite{MR685425}), $R_{\mu}$ have an explicit description as a quotient of the coinvariant ring $$R_{n}=\bb{C}[x_{1},\ldots,x_{n}]/\langle (x_{1},\ldots,x_{n})^{\mf{S}_{n}}\rangle.$$ Let $DR_{n}$ be the ring of coinvariants for the diagonal action of the symmetric group $\mf{S}_{n}$ on $\bb{C}^{n}\oplus \bb{C}^{n}$ defined by $$DR_{n}=\bb{C}[\mb{x},\mb{y}]/\langle(\mb{x},\mb{y})\cap \bb{C}[\mb{x},\mb{y}]^{\mf{S}_{n}}\rangle.$$ Here $\bb{C}[\mb{x},\mb{y}]=\bb{C}[x_{1},y_{1},\ldots ,x_{n},y_{n}]$ is the ring of polynomial functions on $\bb{C}^{n}\oplus \bb{C}^{n}$ and the symmetric group acts diagonally. Then $DR_{n}$ can be considered to be a doubled analogue of $R_{n}$. This is related to $q,t$-analogue of Kostka-Foulkes polynomials or Macdonald polynomials.

Let $\tilde{H}_{\mu}$ be the modified Macdonald polynomials (see e.g. \cite{MR2051783} for the definition). The $\tilde{H}_{\mu}$'s form a basis of the ring of symmetric polynomials with coefficients in $\bb{Q}(q,t)$. Let $\nabla$ be the linear operator defined in terms of the modified Macdonald polynomials by $$\nabla \tilde{H}_{\mu}=t^{n(\mu)}q^{n(\mu')}\tilde{H}_{\mu}.$$ Here, $n(\mu)=\sum_{i}(i-1)\mu_{i}$ and $\mu'$ is the conjugate of $\mu$. 

The $\mf{S}_{n}$-action on $DR_{n}$ respects the bigrading $$DR_{n}=\oplus_{r,s}(DR_{n})_{r,s}$$ given by the x- and y-degrees. Hence we can consider its bigraded Frobenius series $\mca{F}(DR_{n},z;q,t)$. The following theorem proved by using the geometry of Hilbert schemes expresses $\mca{F}(DR_{n},z;q,t)$ in terms of $\nabla$.

\begin{thm}[\cite{MR1918676}]
We have $$\mca{F}(DR_{n},z;q,t)=\nabla e_{n}(z).$$ Here $e_{n}(z)$ is the elementary symmetric function of degree $n$.
\end{thm}

An important problem in combinatorics of diagonal coinvariants is to find a combinatorial description of $\mca{F}(DR_{n},z;q,t)$. In \cite{MR2115257}, Haglund, Haiman, Loehr, Remmel, and Ulyanov proposed a combinatorial formula which conjecturally gives the monomial symmetric function expansion of $\mca{F}(DR_{n},z;q,t)$. We briefly recall their description below.

Let $$\delta_{n}=(n-1,n-2,\ldots ,1,0)$$ be the staircase partition. Let $\la\subset\delta_{n}$ be a partition. Let $T\in\mr{SSYT}(\la +(1^{n})/\la)$ be a semistandard tableau of skew shape $\la +(1^{n})/\la$. For every box $x=(i,j)\in \bb{N}\times \bb{N}$, we set $d(x)=i+j$. Given two entries $T(x)=a$ and $T(y)=b$ with $a<b$, $x=(i,j)$, $y=(i',j')$. We say that these two entries form a $d$-{\it inversion} if either
\begin{enumerate}
\item $d(y)=d(x)$ and $j>j'$ or
\item $d(y)=d(x)+1$ and $j<j'$.
\end{enumerate}

We set $\mr{dinv}(T)$ to be the number of $d$-inversions of $T$. See Figure 1 for an example.

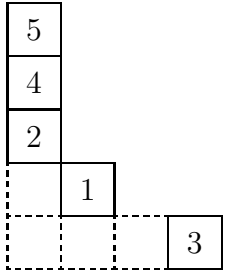
\begin{figure}
\begin{center}
\setlength{\unitlength}{1pt}
\begin{picture}(100,100)

\put(0.2,80.4){\framebox(19.6,19.8){5}}
\put(0.2,60.3){\framebox(19.6,19.8){4}}
\put(0.2,40.2){\framebox(19.6,19.8){2}}
\put(20.2,20.2){\framebox(19.6,19.6){1}}
\put(60.2,0.2){\framebox(19.8,19.6){3}}

\put(0,20){\dashbox{2}(20,20){}}
\put(0,0){\dashbox{2}(20,20){}}
\put(20,0){\dashbox{2}(20,20){}}
\put(40,0){\dashbox{2}(20,20){}}

\end{picture}
\caption{An example of semistandard Young tableau of skew shape $\la+(1^{n})/\la$, where $n=5$ and $\la=(3,1)$. In this case, d-inversions are the pairs $(1,2),(1,3),(2,3),(3,4)$. Hence $\mr{dinv}=4$.}
\end{center}
\end{figure}

\begin{dfn}
For $\la\subset\delta_{n}$, we set $$D^{\la}_{n}(z;q)=\sum_{T\in \mr{SSYT}(\la +(1^{n})/\la)}q^{\mr{dinv}(T)}z^{T},$$ and $$D_{n}(z;q,t)=\sum_{\la\subset\delta_{n}}t^{|\delta_{n}/\la|}D^{\la}_{n}(z;q).$$
\end{dfn}

\begin{conj}[\cite{MR2115257}]
We have an identity $$\mca{F}(DR_{n},z;q,t)=D_{n}(z;q,t).$$
\end{conj}

It is proved in \cite{MR2115257} that $D^{\la}_{n}(z;q)$ are symmetric and Schur positive. Their proof uses the theory of LLT polynomials (\cite{MR1434225}, \cite{MR1864481}) and their Schur positivity. They also conjectured a similar formula for $\nabla^{m}e_{n}(z)$. 

In this paper, we give a geometric interpretation of $D^{\la}_{n}(z;q)$ and $D_{n}(z;q,t)$ by using the geometry of affine analogue of Springer fibers of type $A$. This provides another proof for Schur positivity of $D^{\la}_{n}(z;q)$. 

Let $\alpha_{1},\ldots,\alpha_{n-1}$ be the simple roots of $\mf{sl}_{n}(\bb{C})$ associated to the Borel subalgebra whose elements consist of upper triangular matrices. Let $\theta=\alpha_{1}+\cdots+\alpha_{n-1}$. For each root $\alpha$, we take a nonzero element $e_{\alpha}$ in the root space attached to $\alpha$. We set $$v=\epsilon^{2}e_{-\theta}+\epsilon\sum_{i=1}^{n-1}e_{\alpha_{i}}\in\mf{sl}_{n}[[\epsilon]].$$ This is a regular semisimple nil elliptic element. Let $\hat{\mca{B}}_{v}$ be the set of Iwahori subalgebras in $\mf{sl}_{n}(\!(\epsilon)\!)$ containing $v$. By Kazhdan and Lusztig \cite{MR947819}, $\hat{\mca{B}}_{v}$ has a structure of algebraic variety over $\bb{C}$. Let $\hat{\mca{B}}$ and $X$ be the affine flag variety and affine Grassmannian of $SL_{n}$, and let $\pi:\hat{\mca{B}}\rightarrow X$ be the natural projection. The subvariety $\hat{\mca{B}}_{v}\subset\hat{\mca{B}}$ is closed. Let $X_{v}$ be the image of $\hat{\mca{B}}_{v}$ under $\pi$. 

By Goresky, Kottwitz, and MacPherson \cite{MR2209851}, we have a paving of $X_{v}$ by affine spaces. Each cell is constructed as an intersection of $X_{v}$ and an Iwahori orbit of $X$. We show that nonempty cells are parametrized by Young tableaux contained in $\delta_{n}$ (see Proposition 4.8).

Let $C_{\la}$ be the cell corresponding to $\la\subset\delta_{n}$. Then $\mf{S}_{n}$ acts on the Borel-Moore homology of $$\hat{\mca{B}}_{v,\la}=\pi^{-1}(C_{\la})$$ (see section 2). We also have an affine paving on $\hat{\mca{B}}_{v,\la}$. Hence the Borel-Moore homology $H^{\mr{BM}}_{i}(\hat{\mca{B}}_{v,\la})$ vanishes for all odd $i$. 

We define a grading on $H^{\mr{BM}}_{\ast}(\hat{\mca{B}}_{v,\la})$ by declaring $H^{\mr{BM}}_{2i}(\hat{\mca{B}}_{v,\la})$ to be of degree $i$. The $\mf{S}_{n}$-action preserves the grading. Hence we can consider its Frobenius series $\mca{F}(H^{\mr{BM}}_{\ast}(\hat{\mca{B}}_{v,\la}),z;q)$. 

\begin{thm}[Theorem 4.13 for $m=1$ and $b=1$]
We have $$\mca{F}(H^{\mr{BM}}_{\ast}(\hat{\mca{B}}_{v,\la}),z;q)=q^{\binom{n}{2}}\omega D^{\la}_{n}(z;q^{-1}).$$ Here $\omega$ denotes the standard involution on the ring of symmetric polynomials.
\end{thm}

Let $c$ be a nonnegative integer. We set $$Y_{\leq c}=\bigsqcup_{\substack{\la\subset\delta_{n} \\ |\delta_{n}/\la|\leq c}}\hat{\mca{B}}_{v,\la}.$$ We show that $Y_{\leq c}$ is a closed subvariety of $\hat{\mca{B}}_{v}$ (see Corollary 4.6). Hence $Y_{\leq c}$'s form a stratification of $\hat{\mca{B}}_{v}$. This gives a filtration on $H^{\mr{BM}}_{\ast}(\hat{\mca{B}}_{v})$. This filtration is stable under the $\mf{S}_{n}$-action. Therefore, we obtain a bigraded $\mf{S}_{n}$-module $\mr{gr}_{\ast}H^{\mr{BM}}_{\ast}(\hat{\mca{B}}_{v})$ by taking the associated graded. 

\begin{thm}[Theorem 4.15 for $m=1$ and $b=1$]
We have $$\mca{F}(\mr{gr}_{\ast}H^{\mr{BM}}_{\ast}(\hat{\mca{B}}_{v}),z;q,t)=q^{\binom{n}{2}}\omega D_{n}(z;q^{-1},t).$$
\end{thm}

Actually, we compute the Borel-Moore homology of more general affine Springer fibers of type $A$ associated to regular semisimple nil elliptic elements. As a result, we obtain a generalization of $D^{\la}_{n}(z;q)$ and $D_{n}(z;q,t)$ (see section 3). 

In \cite{MR1486037}, Sommers determined the cohomology of affine Springer fibers associated to some regular semisimple nil elliptic homogeneous elements as $\mf{S}_{n}$-modules without grading. Our main theorem refines the result of Sommers in the case of type $A$. 

We remark that if we can provide another method for calculating the homology of the affine Springer fibers, we would get a formula expressing $D_{n}(z;q,t)$. Let $\pi_{v}:\hat{\mca{B}}_{v}\rightarrow X_{v}$ be the natural projection. Then, the fibers of $\pi_{v}$ are classical Springer fibers, whose homology can be described in terms of Kostka-Foulkes polynomials. Hence if we can describe the stratification of $C_{\la}$ defined by the form of the fibers of $\pi$, then we can express $D^{\la}_{n}(z;q,t)$ in terms of Kostka-Foulkes polynomials. This gives a geometric interpretation of some known results such as Theorem 6.8 in \cite{MR2371044}. However, it seems to be a difficult problem to describe the stratification in general.  

Let us describe the organization of the paper. Section 2 and section 3.1 are review of some known results. In section 2.1, we briefly review a construction of Weyl group action on the homology of Springer fibers. In section 2.2, we review a construction of affine Weyl group action on the homology of affine Springer fibers following \cite{MR1390751}. In section 2.3, we review the result of \cite{MR2209851} for special cases we need. In section 3.1, we introduce some notation on combinatorics and review some facts about quasi-symmetric functions and LLT polynomials from \cite{MR2115257}. In section 3.2, we introduce a generalization of $D^{\la}_{n}(z;q)$ and $D_{n}(z;q,t)$. In section 4, we prove our main theorem. 

The author deeply thanks Professor Eric Sommers for suggesting the problem and sharing his insights.

\section{Springer Theories}

\subsection{Weyl group actions: finite case}\mbox{}

First, we introduce some sheaf theoretic notation. Let $f:X\rightarrow Y$ be a morphism between complex algebraic varieties. We denote by $D^{b}_{c}(X)$ the bounded derived category of constructible sheaves on $X$ and let $f_{\ast},f_{!}:D^{b}_{c}(X)\rightarrow D^{b}_{c}(Y)$, $f^{\ast},f^{!}:D^{b}_{c}(Y)\rightarrow D^{b}_{c}(X)$ be the usual derived functors. Note that these sheaf-theoretic functors are understood to be derived without $R$ or $L$ in the front. We denote by $\bb{C}_{X}$ the constant sheaf on $X$ and  denote by $\bb{D}_{X}=p^{!}\bb{C}_{\mr{Spec}(\bb{C})}$ the dualizing complex of $X$, where $p$ is the natural map from $X$ to $\mr{Spec}(\bb{C})$. 

The $i$-th Borel-Moore homology $H^{\mr{BM}}_{i}(X)$ of $X$ is defined by $$H^{\mr{BM}}_{i}(X)=\bb{H}^{-i}(X,\bb{D}_{X}),$$ where we denote by $\bb{H}^{\ast}$ the functor of taking hypercohomology. 

Let $Y\hookrightarrow  X$ be an embedding of a smooth connected locally closed subvariety of complex dimension $d$ and $\bar{Y}$ the closure of $Y$, and let $\mca{L}$ be a local system on $Y$. Let $j:Y\rightarrow\bar{Y}$ and $i:\bar{Y}\rightarrow X$ be the natural inclusion. We denote by $IC(Y,\mca{L})$ the shift of the intermediate extension $i_{\ast}j_{!\ast}\mca{L}[d]$ of $\mca{L}$ to $X$.

Let $Z$ be a smooth complex algebraic variety, $X$ be an irreducible complex algebraic variety, and $\pi:Z\rightarrow X$ be a proper morphism. We fix a stratification $X=\sqcup\mca{O}_{x}$ of $X$ into a finite number of smooth irreducible subvarieties such that the restriction $\pi_{x}:\pi^{-1}(\mca{O}_{x})\rightarrow\mca{O}_{x}$ of $\pi$ to $\pi^{-1}(\mca{O}_{x})$ is a locally trivial topological fibration for each $\mca{O}_{x}$. It is known that such a stratification exists. 

Let $\mca{O}_{0}$ be the dense open stratum of $X$. We fix a point $x\in\mca{O}_{x}$ in each stratum. Then the fiber $\pi^{-1}(x)$ of each point of $\mca{O}_{x}$ is independent of the choice of $x\in\mca{O}_{x}$. 

We set $d_{x}=\dim\pi^{-1}(x)$ and $c_{x}=\dim X-\dim\mca{O}_{x}$. Then $\pi$ is called {\it semismall} if it satisfies $2d_{x}\leq c_{x}$ for each $\mca{O}_{x}$, and it is called {\it small} if it is semismall and for each $\mca{O}_{x}\ne\mca{O}_{0}$, we have $2d_{x}<c_{x}$. 

If $\pi$ is semismall, then the fiber of each point of $\mca{O}_{0}$ must be zero dimensional. Hence the direct image sheaf $$\mca{L}:=(\pi_{0})_{\ast}\bb{C}_{\pi^{-1}(\mca{O}_{0})}$$ of the constant sheaf on $\pi^{-1}(\mca{O}_{0})$ is concentrate on zero-th degree and is a local system on $\mca{O}_{0}$. Note also that $\dim Z=\dim X$.

\begin{lem}[\cite{MR1021493}]
Let $n=\dim Z$. If $\pi$ is small, then we have $$\pi_{\ast}\bb{C}_{Z}[n]=IC(\mca{O}_{0},\mca{L}).$$
\end{lem}

Let $G$ be a connected and simply connected reductive algebraic group over $\bb{C}$. Fix a Borel subgroup $B\subset G$ and a Cartan subgroup $T\subset B$. Let $\mf{g}$ and $\mf{b}$ be the Lie algebras of $G$ and $B$, respectively. Let $W$ be the Weyl group of $G$. Let $X^{\ast}(T)$ and $X_{\ast}(T)$ be the character and cocharacter lattices of $T$. Let $\Delta$ and $\check{\Delta}$ be the set of roots and coroots of $(G,T)$. Let $\mf{g}=\oplus_{\alpha\in\{0\}\cup\Delta}\mf{g}_{\alpha}$ be the root space decomposition of $\mf{g}$. We fix for each $\alpha\in\Delta$ a nonzero element $e_{\alpha}\in\mf{g}_{\alpha}$.

Let $$\tilde{\mf{g}}=\{(x,g\cdot B)\in \mf{g}\times G/B\mid \mr{Ad}(g)^{-1}x\in \mf{b}\}.$$ Let $\pi:\tilde{\mf{g}}\rightarrow\mf{g}$, where $\pi(x,g\cdot B)=x$, be the natural projection. Then it is known that $\pi:\tilde{\mf{g}}\rightarrow\mf{g}$ is a small morphism with the dense open stratum of $\mf{g}$ being the set of regular semisimple elements $\mf{g}_{\mr{rs}}$. 

Moreover, the restriction $$\pi_{0}:\tilde{\mf{g}}_{\mr{rs}}:=\pi^{-1}(\mf{g}_{\mr{rs}})\rightarrow\mf{g}_{\mr{rs}}$$ of $\pi$ to $\pi^{-1}(\mf{g}_{\mr{rs}})$ is known to be a principal $W$-bundle over $\mf{g}_{\mr{rs}}$. Hence $W$ acts on $$\mca{L}=(\pi_{0})_{\ast}\bb{C}_{\tilde{\mf{g}}_{\mr{rs}}}.$$ 

By functoriality, $W$ also acts on $IC(\mf{g}_{\mr{rs}},\mca{L})=\pi_{\ast}\bb{C}_{\tilde{\mf{g}}}$. By taking fiber at $x\in\mf{g}$, we get a $W$-action on the cohomology of Springer fibers $H^{\ast}(\pi^{-1}(x))$. 

This is the construction of Springer representations due to Lusztig, Borho and MacPherson. Note that $\mca{L}$ is $G$-equivariant by the $G$-equivariance of $\pi$ and the $W$-action is compatible with the $G$-equivariant structure. 

We also need parabolic versions of the above construction. Let $P$ be a parabolic subgroup of $G$ and $M$ be a Levi subgroup of $P$ such that $B\subset P$ and $T\subset M$. Let $W_{P}=N_{M}(T)/T$ be the Weyl group of $M$. Let $\mf{p}$ and $\mf{m}$ be the Lie algebras of $P$ and $M$ respectively. 

We set $$\tilde{\mf{g}}^{P}=\{(x,g\cdot P)\in\mf{g}\times G/P\mid \mr{Ad}(g)^{-1}x\in\mf{p}\}.$$  Let $\pi^{P}:\tilde{\mf{g}}\rightarrow\tilde{\mf{g}}^{P}$ be the morphism defined by $\pi^{P}(x,g\cdot B)=(x,g\cdot P)$. Let $\pi_{P}:\tilde{\mf{g}}_{P}\rightarrow\mf{g}$ be the morphism defined by $\pi_{P}(x,g\cdot P)=x$. 

We also set $\tilde{\mf{g}}^{P}_{\mr{rs}}=\pi_{P}^{-1}(\mf{g}_{\mr{rs}})$ and $$\mca{L}_{P}=(\pi_{P,0})_{\ast}\bb{C}_{\tilde{\mf{g}}^{P}_{\mr{rs}}}.$$ Here $\pi_{P,0}:\tilde{\mf{g}}^{P}_{\mr{rs}}\rightarrow\mf{g}_{\mr{rs}}$ is the restriction of $\pi_{P}$ to $\tilde{\mf{g}}^{P}_{\mr{rs}}$.

Then $\pi_{P}$ is a small morphism. Hence we have $IC(\mf{g}_{\mr{rs}},\mca{L}_{P})=(\pi_{P})_{\ast}\bb{C}_{\tilde{\mf{g}}^{P}}$. Since the restriction $\tilde{\mf{g}}_{\mr{rs}}\rightarrow\tilde{\mf{g}}^{P}_{\mr{rs}}$ of $\pi^{P}$ to $\tilde{\mf{g}}_{\mr{rs}}$ is a principal $W_{P}$-bundle, we have $$\mca{L}_{P}=\mca{L}^{W_{P}}.$$ Here the superscript $W_{P}$ denotes $W_{P}$-invariants. Hence by applying the functor of intermediate extension, we obtain the following.

\begin{prop}[\cite{MR737927}]
With notation being as above, we have$$(\pi_{P})_{\ast}\bb{C}_{\tilde{\mf{g}}^{P}}\cong\left(\pi_{\ast}\bb{C}_{\tilde{\mf{g}}}\right)^{W_{P}}.$$
\end{prop}

\subsection{Affine Springer fibers}\mbox{}

Let $\mca{O}=\bb{C}[[\epsilon]]$ be the ring of formal power series over $\bb{C}$. Let $F=\bb{C}(\!(\epsilon )\!)$ be the field of Laurent series over $\bb{C}$. Let $I\subset G(F)$ be the Iwahori subgroup which is defined as the inverse image of the Borel subgroup $B$ under the projection $G(\mca{O})\rightarrow G$, $\epsilon\mapsto 0$. Let $\hat{\mf{b}}$ be the Lie algebra of $I$.

Let $W_{\mr{aff}}=\Lam\rtimes W$ be the affine Weyl group of $G$. For $\lam \in \Lam$, we write $t_{\lam}=\left(\lam,e\right)\in W_{\mr{aff}}$. Here $e$ denotes the unit in $W$. We will abbreviate $w=\left(0, w\right)\in W_{\mr{aff}}$ for each $w\in W$. 

We embed $\Lam$ into $T(F)\subset G(F)$ by $\lam \mapsto \lam (\epsilon)=:\epsilon^{\lam}$, where we denote by $\lam (\epsilon)$ the image of $\epsilon $ under the map $$F^{\times}=\bb{G}_{m}(F) \rightarrow T(F)$$ induced by $\lam$. 

The affine Weyl group $W_{\mr{aff}}$ is identified with the quotient $N_{G(F)}(T(F))/T(\mca{O})$ of the normalizer of $T(F)$ in $G(F)$ by $T(\mca{O})$, and  $t_{\lam}$ is identified with $\epsilon ^{\lam}$ in this identification. 

Let $X=G(F)/G(\mca{O})$ be the affine Grassmannian and $\hat{\mca{B}}=G(F)/I$ the affine flag variety. These are known to be equipped with a structure of ind-projective ind-schemes over $\bb{C}$. 

An element $\gamma\in\mf{g}(F)$ is called nil if $\mr{Ad}(\gamma)^{N}\rightarrow 0$ as $N\rightarrow\infty$. 

Let $\gamma$ be a nil element. We denote by $$X_{\gamma}=\{g\cdot G(\mca{O})\in X\mid \mr{Ad}(g)^{-1}(\gamma)\in\mf{g}[[\epsilon]]\}$$ and $$\hat{\mca{B}}_{\gamma}=\{g\cdot I\in\hat{\mca{B}}\mid \mr{Ad}(g)^{-1}(\gamma)\in\hat{\mf{b}}\}.$$ We view $X_{\gamma}$ and $\hat{\mca{B}}_{\gamma}$ as ind-schemes over $\bb{C}$ by giving them the reduced ind-scheme structures. Both of $X_{\gamma}$ and $\hat{\mca{B}}_{\gamma}$ are called affine Springer fibers in the literature. 

We denote by $\pi:\hat{\mca{B}}_{\gamma}\rightarrow X_{\gamma}$ the natural projection. Note that the fibers of $\pi$ are classical Springer fibers.

Let $\gamma\in\mf{g}(F)$ be a regular semisimple nil element. It follows that its centralizer $Z_{G(F)}(\gamma)$ in $G(F)$ is a maximal torus. We set $\Lambda_{\gamma}=\Hom_{F}(\bb{G}_{m}, Z_{G(F)}(\gamma))$. A regular semisimple element $\gamma$ is called elliptic if it satisfies $\Lambda_{\gamma}=0$.

\begin{prop}[\cite{MR947819}]
Let $\gamma\in\mf{g}[[\epsilon]]$ be a nil element. Then $X_{\gamma}$ is finite dimensional if and only if $\gamma$ is regular semisimple.
\end{prop}

\begin{prop}[\cite{MR947819}]
Let $\gamma\in\mf{g}[[\epsilon]]$ be a regular semisimple nil element. Then $\Lambda_{\gamma}$ acts freely on $X_{\gamma}$ and $\hat{\mca{B}}_{\gamma}$, and the quotients $\Lambda_{\gamma}\backslash X_{\gamma}$ and $\Lambda_{\gamma}\backslash \hat{\mca{B}}_{\gamma}$ are projective over $\bb{C}$. In particular, if $\gamma$ is elliptic, then $X_{\gamma}$ and $\hat{\mca{B}}_{\gamma}$ are projective.
\end{prop}

\begin{exa}
Let $G=SL_{2}$ and $\gamma =\epsilon^{2}e_{-\alpha}+\epsilon e_{\alpha}$ where $\alpha$ is the only positive simple root. Then $\gamma$ is a regular semisimple nil elliptic element. In this case, $X_{\gamma}$ turns out to be isomorphic to a projective line and $\hat{\mca{B}}_{\gamma}$ is isomorphic to two projective lines intersecting transversally at a single point. The natural projection $\pi:\hat{\mca{B}}_{\gamma}\rightarrow X_{\gamma}$ is a morphism which maps isomorphically on one projective line and maps the other projective line into a point.
\end{exa}

Now we construct an action of affine Weyl group $W_{\mr{aff}}$ on $H^{\mr{BM}}_{\ast}(\hat{\mca{B}}_{\gamma})$. Let $\{\alpha_{i}\}_{i\in I_{\mr{aff}}}$ be the set of affine simple roots. 

For $i\in I_{\mr{aff}}$, let $\hat{\mf{b}}^{i}=\bb{C}e_{-\alpha_{i}}+\hat{\mf{b}}\subset\mf{g}(F)$ be a parahoric subalgebra. Given a subset $\emptyset\ne J\subsetneq I_{\mr{aff}}$, we denote by $\hat{\mf{b}}^{J}$ the Lie subalgebra of $\mf{g}(F)$ generated by $\sum_{i\in J}\hat{\mf{b}}^{i}$. We also denote $\hat{\mf{b}}$ by $\hat{\mf{b}}^{\emptyset}$. 

For $J\subsetneq I_{\mr{aff}}$, let $\hat{B}^{J}=\{g\in G(F)\mid \mr{Ad}(g)(\hat{\mf{b}}^{J})=\hat{\mf{b}}^{J}\}$. This is a proalgebraic group over $\bb{C}$ whose prounipotent radical is denoted by $\hat{U}^{J}$. The quotient $\bar{G}^{J}=\hat{B}^{J}/\hat{U}^{J}$ is a connected reductive algebraic group over $\bb{C}$. Let $\hat{\mf{n}}^{J}$ and $\bar{\mf{g}}^{J}$ be the Lie algebras of $\hat{U}^{J}$ and $\bar{G}^{J}$ respectively. Let $\hat{\mca{B}}^{J}=G(F)/\hat{B}^{J}$ and $\pi^{J}:\hat{\mca{B}}\rightarrow\hat{\mca{B}}^{J}$ be the natural projections.

For any $l\geq 1$, let $$\hat{\mca{B}}^{J}_{\gamma}(l)=\{g\cdot\hat{B}^{J}\in\hat{\mca{B}}^{J}\mid \gamma\in\mr{Ad}(g)\hat{\mf{b}}^{J}\mbox{ and }\epsilon^{l}\hat{\mf{n}}^{\emptyset}\subset\mr{Ad}(g)\hat{\mf{b}}^{J}\}.$$ This is a projective algebraic variety over $\bb{C}$. 

We have a principal $\bar{G}^{J}$-bundle $E\rightarrow\hat{\mca{B}}^{J}_{\gamma}(l)$, where $$E=\{g\cdot\hat{U}^{J}\in G(F)/\hat{U}^{J}\mid g\cdot\hat{B}^{J}\in\hat{\mca{B}}^{J}_{\gamma}(l)\},$$ and the map $E\rightarrow\hat{\mca{B}}^{J}_{\gamma}(l)$ is $g\cdot\hat{U}^{J}\mapsto g\cdot\hat{B}^{J}$.

The Lie algebra bundle $\bar{G}^{J}\backslash(E\times\bar{\mf{g}}^{J})$ associated to $E\rightarrow\hat{\mca{B}}^{J}_{\gamma}(l)$ is the bundle whose fiber at $\hat{\mf{p}}=g\cdot\hat{B}^{J}$ is the Lie algebra $\bar{\mf{p}}=\mr{Ad}(g)\hat{\mf{b}}^{J}/\mr{Ad}(g)\hat{\mf{n}}^{J}$. This bundle has a natural section $j$ given by associating to any $\hat{\mf{p}}\in\hat{\mca{B}}^{J}_{\gamma}(l)$ the image of $\gamma$ under the natural projection $\mr{Ad}(g)\hat{\mf{b}}^{J}\rightarrow\bar{\mf{p}}$.

Let $$\hat{\mca{B}}_{\gamma}(l)=\{g\cdot I\in\hat{\mca{B}_{\gamma}}\mid \epsilon^{l}\hat{\mf{n}}^{\emptyset}\subset\mr{Ad}(g)\hat{\mf{b}}^{J}\}.$$ Then we have the following commutative diagram with each square being cartesian:

$$
\begin{CD}
\hat{\mca{B}}_{\gamma}(l) @>>> \bar{G}^{J}\backslash(E\times\widetilde{\bar{\mf{g}}^{J}})
@<<< E\times\widetilde{\bar{\mf{g}}^{J}} @>>>\widetilde{\bar{\mf{g}}^{J}} \\
@V\pi^{J}VV @V\pi'_{J}VV @V1\times\pi_{J}VV @V\pi_{J}VV \\
\hat{\mca{B}}^{J}_{\gamma}(l) @>j>> \bar{G}^{J}\backslash(E\times\bar{\mf{g}}^{J}) 
@<u<< E\times\bar{\mf{g}}^{J} @>q>> \bar{\mf{g}}^{J}.
\end{CD}
$$

Here, $\pi_{J}:\widetilde{\bar{\mf{g}}^{J}}\rightarrow\bar{\mf{g}}^{J}$ is the Grothendieck alteration for $\bar{G}^{J}$, $\pi'_{J}$ and $1\times\pi_{J}$ are the morphisms naturally induced from $\pi_{J}$, $q$ is the natural projection, and $u$ is the natural quotient map. 

The Weyl group of $\bar{G}^{J}$ is naturally identified with the subgroup $W_{J}$ of $W_{\mr{aff}}$ generated by $\{s_{i}\mid i\in J\}$, where $s_{i}$ is the simple reflection corresponding to $i$. Hence by the previous section, we have an action of $W_{J}$ on $\mca{L}_{J}=(\pi_{J})_{!}\bb{C}_{\widetilde{\bar{\mf{g}}^{J}}}$. Moreover, this action is compatible with the $\bar{G}^{J}$-equivariant structure. 

By the proper base change theorem, we have $$u^{\ast}(\pi'_{J})_{!}\bb{C}_{\bar{G}^{J}\backslash(E\times\widetilde{\bar{\mf{g}}^{J}})}\simeq (1\times\pi_{J})_{!}\bb{C}_{E\times\widetilde{\bar{\mf{g}}^{J}}}\simeq q^{\ast}\mca{L}_{J}$$ and $$\pi^{J}_{!}\bb{C}_{\hat{\mca{B}}_{\gamma}(l)}\simeq j^{\ast}(\pi'_{J})_{!}\bb{C}_{\bar{G}^{J}\backslash(E\times\widetilde{\bar{\mf{g}}^{J}})}.$$ 

Therefore, $W_{J}$ acts on $u^{\ast}(\pi'_{J})_{!}\bb{C}_{\bar{G}^{J}\backslash(E\times\widetilde{\bar{\mf{g}}^{J}})}$. Since this action is compatible with $\bar{G}^{J}$-equivariant structure, $W_{J}$ also acts on $(\pi'_{J})_{!}\bb{C}_{\bar{G}^{J}\backslash(E\times\widetilde{\bar{\mf{g}}^{J}})}$. Hence $W_{J}$ acts on $\pi^{J}_{!}\bb{C}_{\hat{\mca{B}}_{\gamma}(l)}$. By taking compact support cohomology, we get a $W_{J}$-action on $H^{\ast}_{c}(\hat{\mca{B}}_{\gamma}(l))$ and hence on $H^{\mr{BM}}_{\ast}(\hat{\mca{B}}_{\gamma}(l))$ by duality. 

We can see that the closed embedding $\hat{\mca{B}}_{\gamma}(l)\hookrightarrow\hat{\mca{B}}_{\gamma}(l+1)$ induces a map $H^{\mr{BM}}_{\ast}(\hat{\mca{B}}_{\gamma}(l))\rightarrow H^{\mr{BM}}_{\ast}(\hat{\mca{B}}_{\gamma}(l+1))$ which is compatible with the $W_{J}$-actions. Hence we get a $W_{J}$-action on the direct limit $$\varinjlim_{l}H^{\mr{BM}}_{\ast}(\hat{\mca{B}}_{\gamma}(l))=H^{\mr{BM}}_{\ast}(\hat{\mca{B}}_{\gamma}).$$

Now let $J\subset J'$ be two subsets of $I_{\mr{aff}}$ distinct from $I_{\mr{aff}}$. Then $W_{J}\subset W_{J'}$ and the $W_{J}$-action on $H^{\mr{BM}}_{\ast}(\hat{\mca{B}}_{\gamma})$ is the restriction of the analogous $W_{J'}$-action on $H^{\mr{BM}}_{\ast}(\hat{\mca{B}}_{\gamma})$. Since the Coxeter relations involve at most two simple reflections, we obtain a $W_{\mr{aff}}$-action on $H^{\mr{BM}}_{\ast}(\hat{\mca{B}}_{\gamma})$. Moreover, by Proposition 2.2, we have the following proposition.

\begin{prop}
We have $$\left(\pi^{J}_{!}\bb{C}_{\hat{\mca{B}}_{\gamma}}\right)^{W_{J}}=\bb{C}_{\hat{\mca{B}}^{J}_{\gamma}}.$$
\end{prop}

Let $Y\subset\hat{\mca{B}}^{J}_{\gamma}$ be a locally closed subvariety. Then Proposition 2.6 implies that we have a $W_{J}$-action on $H^{\mr{BM}}_{\ast}((\pi^{J})^{-1}(Y))$ satisfying $H^{\mr{BM}}_{\ast}((\pi^{J})^{-1}(Y))^{W_{J}}\simeq H^{\mr{BM}}_{\ast}(Y)$.

\subsection{Pavings of equivalued affine Springer fibers}\mbox{}

Let $\hat{\mf{g}}=\mf{g}\otimes_{\bb{C}}F$. For $r\in\bb{R}$ and $x\in X_{\ast}(T)\otimes_{\bb{Z}}\bb{R}$, we define a subspace $\hat{\mf{g}}_{x}(r)$ of $\hat{\mf{g}}$ by $$\hat{\mf{g}}_{x}(r)=\bigoplus_{\alpha(x)+m=r}\mf{g}_{\alpha}\epsilon^{m}.$$ 

We set $\hat{\mf{g}}_{x,r}:=\prod_{r'\geq r}\hat{\mf{g}}_{x}(r')\subset \hat{\mf{g}}$ and $\hat{\mf{g}}_{x,r+}:=\prod_{r'>r}\hat{\mf{g}}_{x}(r')$. For $v\in\hat{\mf{g}}_{x,r}$, we denote by $\bar{v}$ the image under the composition of the natural projection $\hat{\mf{g}}_{x,r}\rightarrow \hat{\mf{g}}_{x}(r)$ and the map $\hat{\mf{g}}_{x}(r)\rightarrow \mf{g}$ sending $\epsilon^{m}v'$ to $v'$ for $v'\in \mf{g}$.

We have a connected pro-algebraic subgroup $\hat{G}_{x,r}$ of $G(F)$ with its Lie algebra $\hat{\mf{g}}_{x,r}$ (\cite{MR1253198}). We abbreviate $\hat{G}_{x,0}$ by $\hat{G}_{x}$. Then $\hat{G}_{x}$ is a parahoric subgroup of $G(F)$. Note that the adjoint action of $\hat{G}_{x}$ preserves each subspace $\hat{\mf{g}}_{x,r}$. Let $\hat{G}_{x,0+}$ be the prounipotent radical of $\hat{G}_{x}$. 

For $y\in X_{\ast}(T)\otimes_{\bb{Z}}\bb{R}$, and $v\in\hat{\mf{g}}$, we define $\mca{F}_{y}(v)\subset\mca{F}_{y}:=G(F)/\hat{G}_{y}$ by $$\mca{F}_{y}(v)=\{g\cdot\hat{G}_{y}\in\mca{F}_{y}\mid g^{-1}v\in \hat{\mf{g}}_{y,0}\}.$$ 

We take $x,y\in X_{\ast}(T)\otimes_{\bb{Z}}\bb{R}$ and $s\in\bb{R}$ such that the following three assumptions hold: 

\begin{enumerate}
\item $s\geq 0$;
\item $v\in \hat{\mf{g}}_{x,s}$; and
\item $\bar{v}\in \mf{g}$ is regular semisimple.
\end{enumerate}

We consider the intersections of $\mca{F}_{y}(v)$ and $\hat{G}_{x}$-orbits in $\mca{F}_{y}$. Let $c=\epsilon^{\lam}w$ be a representative of $t_{\lam}w\in W_{\mr{aff}}$ in $N_{G(F)}(T(F))$. Note that $$c\cdot\hat{\mf{g}}_{y,0}=\hat{\mf{g}}_{c\cdot y,0},$$ where $c\cdot y=w\cdot y-\lam$. 

Right multiplication by $c$ induces an isomorphism from $(\hat{G}_{x}\cdot\hat{G}_{c\cdot y}/\hat{G}_{c\cdot y})\cap\mca{F}_{c\cdot y}(v)$ to $$S:=\hat{G}_{x}\cdot c\cdot \hat{G}_{y}/\hat{G}_{y}\cap\mca{F}_{y}(v).$$ 

We set $$\tilde{S}_{0+}:=\{g\cdot(\hat{G}_{x}\cap\hat{G}_{c\cdot y})\in\hat{G}_{x}/(\hat{G}_{x}\cap\hat{G}_{c\cdot y})\mid g^{-1}v\in\hat{\mf{g}}_{c\cdot y,0}+\hat{\mf{g}}_{x,s+}\}.$$ Since $\tilde{S}_{0+}$ is invariant under left multiplication of $\hat{G}_{x,0+}$, we can define the quotient space $$S_{0+}:=\hat{G}_{x,0+}\backslash\tilde{S}_{0+}.$$ In \cite{MR2209851}, it is shown that $S_{0+}$ is isomorphic to a variety called Hessenberg variety.

\begin{thm}[\cite{MR2209851}]
With notation and assumptions being as above, $S_{0+}$ is smooth and $S$ has a structure of iterated affine space bundle over $S_{0+}$. Moreover, if $S_{0+}$ is not empty, the dimension of $S$ is 

\begin{equation*}
\dim(S)=\#\{(\alpha, k)\in\Delta\times\bb{Z}\mid 0\leq \langle x,\alpha\rangle+k<s\mbox{ and }\langle c\cdot y,\alpha\rangle+k<0\}.
\end{equation*}
\end{thm}

In case of type $A$, the Hessenberg variety $S_{0+}$ is a point whenever it is nonempty. We remark that the above description of the paving only depends on $\bar{v}$ and not on $v$. 

\section{Combinatorics}

\subsection{Notation and preliminaries}\mbox{}

We mainly follow the notation of \cite{MR2115257}.

We present a partition by the sequence $\la =(\la_{1},\ldots ,\la_{l}), \la_{1}\geq \la_{2}\geq \cdots \geq \la_{l}>0$, and denote its size by $|\la|=\sum_{i}\la_{i}$. If $|\la |=n$, we write $\la \vdash n$. It is understood that $\la_{i}=0$ for $i>l$. We may also write $\la =(1^{m_{1}},2^{m_{2}},\cdots)$ to indicate the partition with its $m_{i}$ parts equal to $i$. The conjugate partition of $\la$ is denoted by $\la '$ which is defined by $\la'_{i}=\sum_{j\geq i}m_{j}$. 

The Young diagram of $\la$ is the set $\{(i,j)\in \bb{N}\times \bb{N}\mid 0\leq j<\la_{i+1}\}$. One pictures elements $(i,j)\in \bb{N}\times \bb{N}$ as boxes arranged with the $i$-axis vertical and the $j$-axis horizontal. By abuse of notation, we usually write $\la$ both for a partition and its Young diagram. A skew Young diagram $\la/\mu$ is the difference of Young diagrams $\la$ and $\mu$ satisfying $\mu \subseteq \la$. For two partitions $\la$ and $\mu =(\mu_{1},\mu_{2},\cdots )$, we write $\la +\mu:=(\la_{1}+\mu_{1},\la_{2}+\mu_{2},\cdots )$. We also write  $\la -\mu:=(\la_{1}-\mu_{1},\la_{2}-\mu_{2},\cdots )$ for $\mu$ not necessarily being a partition. Note that $\la -\mu$ need not be a partition.

A semistandard Young tableau of (skew) shape $\la$ is a function $T$ from the diagram of $\la$ to the ordered alphabets $\mca{A}_{+}=\{1<2<\cdots \}$, which is weakly increasing on each row of $\la$ and strictly increasing on each column. A semistandard tableau is standard if it is a bijection from $\la$ to $\{1,2,\ldots ,|\la|\}$. We also consider tableaux with negative alphabets $\mca{A}_{-}=\{\bar{1}<\bar{2}<\cdots \}$. We call such a tableau negative Young tableau. A negative Young tableau is called semistandard if it is weakly increasing on each column and strictly increasing on each row. We denote
\begin{align*}
\mr{SSYT}(\la )&=\{\mbox{semistandard tableaux }T:\la\rightarrow\mca{A}_{+}\},\\
\mr{SSYT}_{-}(\la)&=\{\mbox{negative semistandard tableaux }T:\la\rightarrow\mca{A}_{-}\},\\
\mr{SSYT}(\la,\mu)&=\{\mbox{semistandard tableau }T:\la\rightarrow \mca{A}_{+}\mbox{ with entries }1^{\mu_{1}},2^{\mu_{2}},\cdots \},\\
\mr{SSYT}_{-}(\la,\mu)&=\{\mbox{negative semistandard tableau }T:\la\rightarrow\mca{A}_{-}\mbox{ with entries }\bar{1}^{\mu_{1}},\bar{2}^{\mu_{2}},\cdots\},\\
\mr{SYT}(\la )&=\{\mbox{standard tableaux }T:\la \rightarrow \{ 1,\ldots,|\la|\}\},\\
\mr{SYT}_{-}(\la )&=\{\mbox{negative standard tableaux }T:\la \rightarrow \{\bar{1},\ldots,\overline{|\la|}\}\}.
\end{align*}

We write $e_{\la}$ for the elementary symmetric functions, $h_{\la}$ for the complete symmetric functions, $m_{\la}$ for the monomial symmetric functions, $s_{\la}$ for the Schur functions. We take these in variables $z=z_{1},z_{2},\cdots$ or $w=w_{1},w_{2},\cdots$. We adopt the convention that $z_{\bar{a}}$ stands for $w_{a}$ for every negative letter $\bar{a}\in \mca{A}_{-}$. We denote the ring of symmetric polynomials with coefficients in $\bb{Q}$ by $\mr{Sym}$.

We write $\langle -,-\rangle$ for the Hall inner product on $\mr{Sym}$ defined by either of the identities $$\langle h_{\la},m_{\mu}\rangle = \delta _{\la \mu}=\langle s_{\la},s_{\mu}\rangle .$$

We denote by $\omega$ the involution on $\mr{Sym}$ defined by any of the identities
\begin{eqnarray*}
\omega e_{\la}=h_{\la} & \omega h_{\la}=e_{\la} & \omega s_{\la}=s_{\la'}.
\end{eqnarray*}

If $T$ is a (negative) semistandard tableau of (skew) shape $\la$, we set $$z^{T}=\prod_{x\in \la}z_{T(x)}.$$ Note that if $T$ is negative, $z^{T}$ is a monomial of variable $w$ by convention.

Given any subset $D\subset\{1,\ldots,n-1\}$, Gessel's quasi-symmetric function is defined by the formula $$\mca{Q}_{n,D}(z)=\sum_{\substack{a_{1}\leq a_{2}\leq\cdots\leq a_{n} \\ a_{i}=a_{i+1}\Rightarrow i\notin D}}z_{a_{1}}z_{a_{2}}\cdots z_{a_{n}}.$$ Here the indices $\{a_{i}\}$ belong to $\mca{A}_{+}$. We also define a negative version of quasi-symmetric functions $$\tilde{\mca{Q}}_{n,D}(w)=\sum_{\substack{a_{1}\leq a_{2}\leq\cdots\leq a_{n} \\ a_{i}=a_{i+1}\Rightarrow i\in D}}w_{a_{1}}w_{a_{2}}\cdots w_{a_{n}}.$$ 

\begin{lem}[\cite{MR2115257}, Corollary 2.4.3]
Let $f(z)$ be any symmetric function which is homogeneous of degree $n$. Assume that $f(z)$ is written in terms of quasi-symmetric functions as $$f(z)=\sum_{D}c_{D}\mca{Q}_{n,D}(z).$$ Then $\omega f(w)$ is given by $$\omega f(w)=\sum_{D}c_{D}\tilde{\mca{Q}}_{n,D}(w).$$
\end{lem}

For $x=(i,j)\in\bb{N}\times\bb{N}$, we set $c(x)=j-i$. Let $(s_{0},s_{1},\ldots,s_{r-1})$ be an $r$-tuple of integers satisfying $s_{i}\equiv i\mod r$, and let $\boldsymbol{\mu}=(\mu^{(0)},\mu^{(1)},\ldots,\mu^{(r-1)})$ be an $r$-tuple of partitions. For $x\in\mu^{(i)}$, we set $$\tilde{c}(x)=rc(x)+s_{i}.$$ Let $T$ be an semistandard tableau of shape $\boldsymbol{\mu}$, that is, an element of $\mr{SSYT}(\mu^{(0)})\times\mr{SSYT}(\mu^{(1)})\times\cdots\times\mr{SSYT}(\mu^{(r-1)})$. An \textit{inversion} is a pair $(x,y)$ of cells of $\boldsymbol{\mu}$ satisfying $T(x)<T(y)$ and $0<\tilde{c}(x)-\tilde{c}(y)<r$. We denote by $\mr{inv}(T)$ the number of inversions in $T$.

\begin{prop}[\cite{MR2115257}, Corollary 5.2.4]
The following polynomial $$\sum_{T\in\mr{SSYT}(\boldsymbol{\mu})}q^{\mr{inv}(T)}z^{T}$$ is symmetric in $z$.
\end{prop}

It is well-known that irreducible representations of the symmetric group $\mf{S}_{n}$ over $\bb{C}$ are classified by partitions of $n$. We write $L_{\la}$ for the irreducible representation of $\mf{S}_{n}$ corresponding to $\la \vdash n$. Here, we take the convention that $L_{(n)}$ is the trivial representation. Let $V$ be a finite dimensional representation of $\mf{S}_{n}$. We define its Frobenius characteristic $\mca{F}(V,z)\in \mr{Sym}$ by $$\mca{F}(V,z)=\sum_{\la \vdash n}\dim \Hom _{\mf{S}_{n}}(L_{\la},V)s_{\la}(z).$$ 

Let $V$ be a representation of $\mf{S}_{n}$ with a grading $V=\bigoplus _{i=0}^{\infty}V_{i}$ such that each $V_{i}$ is finite dimensional and stable under the action of $\mf{S}_{n}$. We define its Frobenius series $\mca{F}(V,z;q)$ by$$\mca{F}(V,z;q)=\sum_{i=0}^{\infty}q^{i}\mca{F}(V_{i},z).$$ We define bigraded Frobenius series similarly for a $\mf{S}_{n}$-module with a bigrading.

Let $\mu =(\mu_{1},\ldots ,\mu_{l}) \vdash n$ be a partition of $n$. We define the parabolic subgroup $\mf{S}_{\mu}$ of $\mf{S}_{n}$ corresponding to $\mu$ by $$\mf{S}_{\mu}=\mf{S}_{\mu_{1}}\times\mf{S}_{\mu_{2}}\times\cdots\times\mf{S}_{\mu_{l}}.$$

\begin{lem}
Let $V$ be a finite dimensional representation of $\mf{S}_{n}$. Then $$\langle \mca{F}(V,z),h_{\mu} \rangle =\dim V^{\mf{S}_{\mu}}.$$
\end{lem}

\begin{proof}
We may assume $V=L_{\la}$ by semisimplicity of representations of $\mf{S}_{n}$. Then the LHS is
\begin{equation*}
\langle s_{\la},h_{\mu}\rangle =\sum_{\nu \vdash n}K_{\la \nu}\langle m_{\nu},h_{\mu}\rangle =K_{\la \mu},
\end{equation*}
where $K_{\la \mu}$ is the Kostka number.

On the other hand, the RHS is
\begin{align*}
\dim L_{\la}^{\mf{S}_{\mu}}&=\dim \Hom_{\mf{S}_{\mu}}(\mb{triv},L_{\la}) \\
                           &=\dim \Hom_{\mf{S}_{n}}\left( \mr{Ind}_{\mf{S}_{\mu}}^{\mf{S}_{n}}(\mb{triv}),L_{\la}\right) =K_{\la \mu},
\end{align*}
where $\mb{triv}$ is the trivial representation of $\mf{S}_{\mu}$ and we used the fact (see for example \cite{MR1824028}) that $$\mr{Ind}_{\mf{S}_{\mu}}^{\mf{S}_{n}}(\mb{triv})\cong \bigoplus _{\nu \vdash n}L_{\nu}^{\oplus K_{\nu \mu}}.$$  
\end{proof}                             

The above lemma can be used to determine the structure of representations of the symmetric groups since $h_{\mu}$'s form a basis of the space of symmetric polynomials of degree $n$. Except for type $A$, the number of irreducible representations of Weyl groups and the number of parabolic subgroups are distinct. Hence this method does not work for other types.

\subsection{A generalization of HHLRU symmetric functions}\mbox{}

Let $m$ be a nonnegative integer and $b$ an integer such that $1\leq b<n$, $(n,b)=1$. We fix $m$ and $b$ throughout this paper. Let $$\delta =(m(n-1)+b-1,m(n-2)+b-1,\ldots ,m+b-1,b-1).$$  Let $$\delta' =(\delta'_{1},\ldots ,\delta'_{n}),$$ where $\delta'_{l}=l'$ is the integer such that $lb=l'n+l''$, $1\leq l''\leq n$. Note that $\delta -\delta'$ is a partition and its Young diagram consists of all boxes lying below the line through $(0,mn+b)$ and $(n,0)$.

Let $\la\subset(\delta -\delta')$ be a partition. We define $d$-inversion statistics on $\mr{SSYT}(\la+(1^{n})/\la)$ or $\mr{SSYT}_{-}(\la+(1^{n})/\la)$ generalizing the case of $m=1$, $b=1$ recalled in the introduction. For this purpose, we introduce some more notation.

For $x=(i,j)\in \bb{N}\times \bb{N}$, we set $d_{m}(x)=mi+j$ and $$r(x)=(mn+b)n-(mn+b)(i+1)-n(j+1).$$ For two elements $x=(i,j)$, $y=(i',j')\in \{0,1,\ldots,n-1\}\times \bb{N}$ satisfying $i\neq i'$, we set $d(x,y):=d_{m}(y)-d_{m}(x)+l'$, where $l'$ is the integer such that $(i'-i)b=l'n+l''$, $1\leq l''\leq n-1$. We also set $l(x,y)=l''$. We write $x>_{d}y$ if $d(x,y)\geq 0$. 

Since $(i-i')b=(-1-l')n+(n-l'')$, we have $$d(y,x)=d_{m}(x)-d_{m}(y)-1-l'=-d(x,y)-1$$ and $l(y,x)=n-l(x,y)$. Note that since $0\leq i,i'\leq n-1$, we have $l(x,y)\leq n-1$. By $$\frac{r(x)-r(y)}{n}=m(i'-i)+(j'-j)+l'+\frac{l(x,y)}{n},$$ we have $$d(x,y)=\left[ \frac{r(x)-r(y)}{n}\right],$$ where $\left[ r\right]$ is the greatest integer not greater than $r$. It follows from this description of $d(x,y)$ that $x>_{d}y$ if and only if $r(x)>r(y)$. Here, note that $r(x)=r(y)$ implies $x=y$. Hence $>_{d}$ defines a total ordering on $\{0,1,\ldots,n-1\}\times \bb{N}$. 

For $x>_{d}y$, we set
\begin{eqnarray*}
m(x,y)&=&\begin{cases}\mr{max}(0,m+1-d(x,y)) & \mbox{ if }1\leq l(x,y)<b, \\ \mr{max}(0,m-d(x,y)) & \mbox{ if }b\leq l(x,y)<n,\end{cases}\\
n(x,y)&=&\mr{max}(0,m(x,y)-1).
\end{eqnarray*}

We set 
\begin{equation}
A=\{ (x,y)\in\left(\la+(1^{n})/\la\right) \times\left(\la+(1^{n})/\la\right) \mid x>_{d}y,i>i',l(x,y)<b,\mbox{ and }l(y,x)\geq b\}, 
\end{equation}
\begin{equation}
B=\{ (x,y)\in\left(\la+(1^{n})/\la\right) \times\left(\la+(1^{n})/\la\right) \mid x>_{d}y,i>i',l(x,y)\geq b,\mbox{ and }l(y,x)<b\}.
\end{equation}
Note that if $b=1$, both $A$ and $B$ are empty.

Let $T\in \mr{SSYT}(\la+(1^{n})/\la)$ (resp. $T\in \mr{SSYT}_{-}(\la+(1^{n})/\la)$). Let $x=(i,j)$, $y=(i',j')\in \la+(1^{n})/\la$ with $x>_{d}y$. We say that this pair of entries contributes $d$-{\it inversions} $h$ times, where $h\in \bb{Z}$ is determined by the following rules:
\begin{enumerate}
\item if $T(x)<T(y)$ (resp. $T(x)\leq T(y)$), we have 
\begin{equation*}
h=\begin{cases}m(x,y)-1 & \mbox{ if }(x,y)\in A,\\ m(x,y)+1 & \mbox{ if }(x,y)\in B,\\ m(x,y) & \mbox{ if }(x,y)\notin A\cup B.\end{cases}
\end{equation*}
\item if $T(x)\geq T(y)$ (resp. $T(x)>T(y)$), we have 
\begin{equation*}
h=\begin{cases}n(x,y)-1 & \mbox{ if }(x,y)\in A,\\ n(x,y)+1 & \mbox{ if }(x,y)\in B,\\ n(x,y) & \mbox{ if }(x,y)\notin A\cup B.\end{cases}
\end{equation*}
\end{enumerate}
We set $\mr{dinv}(T)$ to be the total number of $d$-inversions. See Figure 2 for an example.

\begin{figure}
\begin{center}
\setlength{\unitlength}{1pt}
\begin{picture}(60,100)

\put(0.2,80.4){\framebox(19.6,19.8){5}}
\put(0.2,60.3){\framebox(19.6,19.8){3}}
\put(0.2,40.2){\framebox(19.6,19.8){2}}
\put(20.2,20.2){\framebox(19.6,19.6){1}}
\put(40.2,0.2){\framebox(19.8,19.6){4}}

\put(0,20){\dashbox{2}(20,20){}}
\put(0,0){\dashbox{2}(20,20){}}
\put(20,0){\dashbox{2}(20,20){}}

\end{picture}
\caption{An example of semistandard Young tableau of skew shape $\la+(1^{n})/\la$, where $n=5, m=0, b=3$ and $\la=(2,1)$. In this case, $(2,1),(1,4)$, and $(3,4)$ belong to $A$, and $(2,4)$ belongs to $B$. The pairs $(2,4),(1,2)$, and $(4,5)$ contribute $1$ d-inversion and $(2,1)$ contributes $-1$ d-inversion. Hence $\mr{dinv}=2$.}
\end{center}
\end{figure}
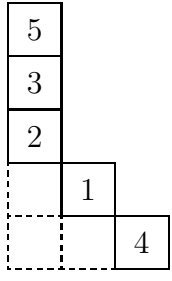

We set $$m(\la)=\mr{max}\{\mr{dinv}(T)\mid T\in \mr{SSYT}(\la +(1^{n})/\la)\}.$$ Note that the same maximum is also attained by a negative semistandard tableau satisfying $T(x)=T(y)\in \mca{A}_{-}$ for any $x,y\in\la+(1^{n})/\la$.

We further define two variants of $d$-inversions. As above, let $T\in \mr{SSYT}(\la+(1^{n})/\la)$ (resp. $T\in \mr{SSYT}_{-}(\la+(1^{n})/\la)$) and $x,y\in\la+(1^{n})/\la$. We say that this pair form a {\it reduced} $d$-{\it inversion} if $T(x)<T(y)$ (resp. $T(x)\leq T(y)$) and the following condition hold:
\begin{equation}
\begin{cases}
0\leq d(x,y)\leq m,\mbox{ if }1\leq l(x,y)<b,\\
0\leq d(x,y)\leq m-1,\mbox{ if }b\leq l(x,y)<n.
\end{cases}
\end{equation}
We set $\mr{dinv}'(T)$ to be the number of reduced $d$-inversions. We define $\mr{dinv}''(T)$ in the same way as $\mr{dinv}'(T)$ by replacing the condition $T(x)<T(y)$ (resp. $T(x)\leq T(y)$) by $T(x)\geq T(y)$ (resp. $T(x)>T(y)$).

\begin{lem}
There is a constant $e(\la)$ depending only on $\la\subset(\delta-\delta')$ such that $$\mr{dinv}(T)=e(\la)+\mr{dinv}'(T)$$ for any $T\in\mr{SSYT}(\la+(1^{n})/\la)$ or $T\in\mr{SSYT}_{-}(\la+(1^{n})/\la)$.
\end{lem}

\begin{proof}
Note that for $x,y\in\la+(1^{n})/\la$, $x>_{d}y$, we have $m(x,y)=n(x,y)+1$ if $x$ and $y$ satisfy (3.3) and $m(x,y)=n(x,y)$ otherwise. 

Hence it suffices to take $e(\la)$ to be the number of $d$-inversions for $T$ satisfying $T(x)>T(y)$ for all $x>_{d}y$. We remark that such tableaux may not be semistandard, but the definition of $\mr{dinv}$ also works for such tableaux. 
\end{proof}

\begin{lem}
For any $T\in\mr{SSYT}(\la+(1^{n})/\la)$ or $T\in\mr{SSYT}_{-}(\la+(1^{n})/\la)$, we have $$m(\la)-e(\la)=\mr{dinv}'(T)+\mr{dinv}''(T).$$
\end{lem}

\begin{proof}
By Lemma 3.4, we have $$m(\la)-e(\la)=\mr{max}\{\mr{dinv}'(T)\mid T\in \mr{SSYT}(\la +(1^{n})/\la)\}.$$ Therefore, $x>_{d}y$ contributes to both sides of the assertion if and only if it satisfies (3.3). This proves the lemma.
\end{proof}

\begin{cor}
For any $T\in\mr{SSYT}(\la+(1^{n})/\la)$ or $T\in\mr{SSYT}_{-}(\la+(1^{n})/\la)$, we have $$\mr{dinv}(T)=m(\la)-\mr{dinv}''(T).$$
\end{cor}

Now we define an analogue of $D^{\la}_{n}(z;q)$ for general $m$, $b$.

\begin{dfn}
For $\la\subset(\delta-\delta')$, we set $$D^{\la}(z;q)=\sum_{T\in\mr{SSYT}(\la+(1^{n})/\la)}q^{\mr{dinv}(T)}z^{T},$$ and $$D(z;q,t)=\sum_{\la\subset(\delta-\delta')}t^{|(\delta-\delta')/\la|}D^{\la}(z;q).$$
\end{dfn}

\begin{prop}[Proved for $b=1$ case in \cite{MR2115257}]
For every $\la\subset(\delta-\delta')$, the polynomial $D^{\la}(z;q)$ is symmetric in $z$.
\end{prop}

\begin{proof}
To prove the assertion, it suffices to identify $D^{\la}(z;q)$ with an polynomial in Proposition 3.2. In the notation of Proposition 3.2, we take $r=mn+b$.

We write $\la=(0^{\alpha_{0}},1^{\alpha_{1}},\ldots,(r-1)^{\alpha_{r-1}})$, where $\alpha_{0}$ is determined by the relation $\sum_{j}\alpha_{j}=n$. Let $\beta$ be the permutation of $\{0,1,\ldots,r-1\}$ such that $$\beta(j)\equiv -nj \mod r.$$ Such $\beta$ exists since $n$ and $r$ are coprime. We set $$\mu^{(\beta(j))}=(1^{\alpha_{j}}).$$ Then we have a natural bijection between cells of $\la+(1^{n})/\la$ and $\boldsymbol{\mu}$ by translating column $j$ of $\la+(1^{n})/\la$ to $\mu^{(\beta(j))}$. 

Let $$s_{\beta(j)}=-n j-(m n+b)\la'_{j+1}.$$ Let $x=(i,j)\in\la+(1^{n})/\la$ and $x'\in\boldsymbol{\mu}$ be the corresponding cell. Then we have $$\tilde{c}(x')=-n j-(m n+b)i.$$

For two distinct cells $x,y\in\la+(1^{n})/\la$ and the corresponding cells $x',y'\in\boldsymbol{\mu}$, we have $$\tilde{c}(x')-\tilde{c}(y')=nd(x,y)+l(x,y).$$ Hence the inequality $$0<\tilde{c}(x')-\tilde{c}(y')<mn+b$$ is equivalent to (3.3). Therefore, if $T\in\mr{SSYT}(\la+(1^{n})/\la)$ and $T'\in\mr{SSYT}(\boldsymbol{\mu})$ are coresponding to each other, then we have $$\mr{dinv}(T)=\mr{inv}(T').$$ By Proposition 3.2, this implies that $$\sum_{T\in\mr{SSYT}(\la+(1^{n})/\la)}q^{\mr{dinv'}(T)}z^{T}$$ is a symmetric function. Therefore, the assertion follows from Lemma 3.4.

\end{proof}

Just as in \cite{MR2115257}, we can prove the following lemma by virtue of the notion of negative Young tableau.

\begin{lem}
We have $$\omega D^{\la}(w;q)=\sum_{T\in\mr{SSYT}_{-}(\la+(1^{n})/\la)}q^{\mr{dinv}(T)}z^{T}.$$
\end{lem}

\begin{proof}
We first expand $D^{\la}(z;q)$ into quasi-symmetric functions. 

Define $a\in\{1,\ldots,n-1\}$ to be a $d$-{\it descent} of a standard tableau $S\in\mr{SYT}(\la +(1^{n})/\la)$ if $S(x)=a$, $S(y)=a+1$ with $x>_{d}y$. We denote by $dd(S)\subset\{1,\ldots,n-1\}$ the set of $d$-descents of $S$. 

Define the standardization $\mr{st}(T)$ of a tableau $T\in\mr{SSYT}(\la+(1^{n})/\la)$ or $T\in\mr{SSYT}_{-}(\la+(1^{n})/\la)$ to be the unique standard tableau $S$ such that $T\circ S^{-1}$ is weakly increasing and if $T\circ S^{-1}(j)=T\circ S^{-1}(j+1)=\cdots=T\circ S^{-1}(k)=b$, then $\{j,\ldots,k-1\}\cap dd(S)$ is empty if $b$ is positive, equal to $\{j,\ldots,k-1\}$ if $b$ is negative. 

As in the proof of \cite{MR2115257} Theorem 3.2.1, we can verify that $$\sum_{\substack{T\in\mr{SSYT}(\la+(1^{n})/\la)\\ \mr{st}(T)=S}}z^{T}=\mca{Q}_{n,dd(S)}(z),$$ and $$\sum_{\substack{T\in\mr{SSYT}_{-}(\la+(1^{n})/\la)\\ \mr{st}(T)=S}}z^{T}=\tilde{\mca{Q}}_{n,dd(S)}(w).$$ 

It is easy to see that $\mr{dinv}(T)=\mr{dinv}(\mr{st}(T))$. Hence we have $$D^{\la}(z;q)=\sum_{S\in\mr{SYT}(\la +(1^{n})/\la)}q^{\mr{dinv}(S)}\mca{Q}_{n,dd(S)}(z),$$ and $$\sum_{T\in \mr{SSYT}_{-}(\la +(1^{n})/\la)}q^{\mr{dinv}(T)}z^{T}=\sum_{S\in\mr{SYT}_{-}(\la +(1^{n})/\la)}\tilde{\mca{Q}}_{n,dd(S)}(w).$$ 

The lemma follows from these formulas and Lemma 3.1 and Proposition 3.8.
\end{proof}

\section{Affine Springer fibers and the HHLRU type formula}

\subsection{Notation}\mbox{}

We keep the notation of section 2 specializing $G$ to be $SL_{n}$.

Let $\Pi=\left\{\alpha_{1},\ldots,\alpha_{n-1}\right\}\subset \Delta$ be the subset of simple roots associated with $B$. Here, we take the convention that roots of $\mf{b}$ are positive. Let $\check{\Pi}=\left\{\al_{1},\ldots,\al_{n-1}\right\}$ be the set of simple coroots, so that $\Lam=\bigoplus_{i=1}^{n-1}\bb{Z}\al_{i}$. We will view $\al_{i}=\ep_{i}-\ep_{i+1}$ as elements of $\left\{\left(x_{1},\ldots ,x_{n}\right)\in\bb{Z}^{n}\mid\Sigma_{i=1}^{n}x_{i}=0\right\}$, where $\ep_{i}=\left(0,\ldots,0,1,0,\ldots,0\right)$ with $1$ in the $i$-th component. 

Let $$\langle-,-\rangle:\Lam\times X^{\ast}(T)\rightarrow\bb{Z}$$ be the canonical pairing. 

We set $\al_{n}=-\al_{1}-\cdots-\al_{n-1}$. For any $l\in\bb{Z}$, we write $\al_{l}=\al_{l'}$, where $l'$ is the integer such that $1\leq l'\leq n$, and $l\equiv l'\mod n$. This notation will simplify some formula. 

We take $\check{\rho}\in\Lam\otimes_{\bb{Z}}\bb{Q}$ such that $\langle\check{\rho},\alpha_{i}\rangle=1$ for any $i$. For $k\in\bb{Z}$, we write $\Delta_{k}=\left\{\alpha\in R\mid\langle\check{\rho},\alpha\rangle=k\right\}.$ 

Let $\mu=(\mu_{1},\ldots,\mu_{l})$ be a composition of $n$, i.e. $\mu_{1}+\cdots+\mu_{l}=n$ and $\mu_{i}\geq 0$ for $i=1,\ldots,l$. We set $$\tilde{\mu}_{k}=\mu_{1}+\cdots+\mu_{k}$$ for $k=1,\ldots,l-1$. Let $\mf{p}_{\mu}$ be the parabolic subalgebra of $\mf{g}$ generated by $\mf{b}$ and $\{e_{-\alpha_{i}}\}_{i\in I_{\mu}}$, where $I_{\mu}=\left\{i\in\left\{1,\ldots,n-1\right\}\mid  i\ne\tilde{\mu}_{k},\mbox{ for any }k\right\}$. Let $P_{\mu}$ the parabolic subgroup corresponding to $\mf{p}_{\mu}$. 

Let $W_{\mu}$ be the subgroup of $W$ generated by $s_{i}$, $i\ne\tilde{\mu}_{k}$, for every $1\leq k<l$. We take $\check{\rho}_{\mu}\in\Lam\otimes_{\bb{Z}}\bb{Q}$ such that $\langle\check{\rho}_{\mu},\alpha_{i}\rangle$ equals to $1$ if $i=\tilde{\mu}_{k}$ for some $k$, and $0$ otherwise. 

Let $s_{\al}\in W$ be the reflection associated to $\al \in \Lam$. For each $i \in \left\{ 1, \ldots, n-1 \right\}$, we write $s_{i}=s_{\al_{i}}$. We set $s_{0}=t_{\al_{0}} s_{\al_{0}}\in W_{\mr{aff}}$ and for any integer $l$, we set $s_{l}=s_{l'}$, where $l'$ is the integer such that $1\leq l'\leq n$, and $l\equiv l'\mod n$. 

Since $G$ is simply connected, $W_{\mr{aff}}$ is a Coxeter group with its set of simple reflections $\left\{ s_{i}\right\}_{i=0,1,\ldots,n-1}$. Let $\ell : W_{\mr{aff}} \rightarrow \bb{Z}_{\geq 0}$ be the length function with respect to these simple reflections.
 
Let $W^{\mr{f}}$ be the set of minimal length representatives of $W_{\mr{aff}}/W$. For each $\lam\in\Lam$, we set $$\ell^{\mr{f}}(\lam) =\mr{min}\left\{\ell(t_{\lam}w)\mid w\in W\right\}.$$ This is the length of the unique element of $W^{\mr{f}}$ contained in the coset $t_{\lam}W$.

Let $\hat{P}_{\mu}$ be the parahoric subgroup corresponding to $\mu$ which are defined as the inverse image of $P_{\mu}$ under the projection $G(\mca{O})\rightarrow G(\bb{C})$, $\epsilon \mapsto 0$. 

Recall that we fixed two positive integers $m$ and $b$ satisfying $1\leq b<n$, and $(n,b)=1$. We set $$v:=\epsilon^{m}(\epsilon \sum_{\alpha\in\Delta_{b-n}}e_{\alpha}+\sum_{\alpha\in \Delta_{b}}e_{\alpha})\in\mf{g}[[\epsilon]].$$ This is a regular semisimple nil elliptic element.

In the notation of section 2.3, we have $v\in\hat{\mf{g}}_{\check{\rho},s}$, where $s=\frac{mn+b}{n}$. We also have $$\bar{v}=\sum_{\alpha\in\Delta_{b}\sqcup\Delta_{b-n}}e_{\alpha}.$$ This is regular semisimple. Hence the assumptions in section 2.3 are satisfied by taking $y=\check{\rho}_{\mu}$, $t=0$, $x=\frac{1}{n}\check{\rho}$, and $s=\frac{mn+b}{n}$. In this case, we have $\hat{G}_{y}=\hat{P}_{\mu}$ and $\hat{G}_{x}=I$. 

We write $$X_{v}=\{g\cdot G(\mca{O})\in G(F)/G(\mca{O})\mid \mr{Ad}(g)^{-1}(v)\in\mf{g}(\mca{O})\},$$ $$\hat{\mca{B}}_{v}=\{g\cdot I\in G(F)/I\mid \mr{Ad}(g)^{-1}(v)\in\mr{Lie}(I)\},$$ and $$\hat{\mca{P}}_{\mu,v}=\{g\cdot\hat{P}_{\mu}\in G(F)/\hat{P}_{\mu}\mid \mr{Ad}(g)^{-1}(v)\in\mr{Lie}(\hat{P}_{\mu})\}$$ for the affine Springer fibers associated to $v$. 

We have a $\bb{G}_{m}$-action on $\hat{\mca{B}}_{v}$. This action is induced from the $\bb{G}_{m}$-action on $\mf{g}(F)$ given by $t\cdot\epsilon^{k}e_{\alpha}=t^{2(nk+\langle\check{\rho},\alpha\rangle)}\epsilon^{k}e_{\alpha}$ for each $t\in\bb{C}^{\times}$, $k\in\bb{Z}$, and $\alpha\in\Delta$. The fixed point set of this action is contained in $\{w\cdot I/I\mid w\in W_{\mr{aff}}\}$.

\subsection{Closure relations}\mbox{}

We set
\begin{equation*}
P:=\left\{ (a_{i})_{i\in \bb{Z}} \mid a_{i}\in \bb{Z}, a_{n}=0, a_{1},\ldots ,a_{n-1}\geq 0, a_{i+n}=a_{i}+1 \right\},
\end{equation*}
which is essentially the set of $(n-1)$-tuple of nonnegative integers. Note that elements of $P$ are determined by $n-1$ components $(a_{1},\ldots ,a_{n-1})$, but we allow the index $i$ to take values on all integers in order to simplify several formulas in this section. For an element $(a_{i})_{i\in \bb{Z}}$ in $P$, let $a((a_{i})_{i\in \bb{Z}}):=a_{1}+\cdots +a_{n-1}$. We often abbreviate $a=a((a_{i})_{i\in \bb{Z}})$ when there is no risk of confusion from the context.

The set $P$ is designed to bridge the gap between group theoretical and combinatorial data. First, we have the following proposition which connects $P$ and $\Lam$.

\begin{prop}
There exists a bijection $\lam : P \rightarrow \Lam$ such that for $(a_{i})\in P$,
\begin{equation} 
\langle \lam ((a_{i})_{i\in \bb{Z}}), \alpha _{l} \rangle = \begin{cases}a_{l-a}-a_{l-a+1} & \mbox{if }l\notin n\bb{Z} \\ a_{l-a}-a_{l-a+1}+1 & \mbox{if }l\in n\bb{Z}. \end{cases}
\end{equation}
\end{prop}

\begin{proof}
We construct the inverse of the claimed bijection by assigning for each element of $\Lam$ an $(n-1)$-tuple of nonnegative integers. Let $x=(x_{1},x_{2},\ldots ,x_{n})$ be an element of $\Lam \subset \bb{Z}^{n}$. 

For $x=(0,\ldots ,0)$, we assign $(a_{1},\ldots ,a_{n-1})=(0,\ldots ,0)$. 

We assume $x\ne (0,\ldots ,0)$. We set $$k=\max\{i\in \{1,2,\ldots ,n\} \mid x_{i}=\min_{j}\{x_{j}\}\}.$$ Since $x\ne (0,\ldots ,0)$ and $\sum_{i}x_{i}=0$, we have $x_{k}\leq -1$. We set $m=-x_{k}-1\geq 0$. We assign $a_{1},a_{2},\ldots ,a_{n-1}$ to $x$ by the following: 
\begin{align}
(x_{1},x_{2},\ldots ,x_{n})=(a_{n-k+1}-m-1,a_{n-k+2}-m-1,\ldots ,a_{n-1}-m-1,\hspace{3em}\\
-m-1,a_{1}-m,a_{2}-m,\ldots ,a_{n-k}-m),\notag
\end{align}
where $-m-1$ is in $k$-th component. By the definition of $k$, we have $a_{i}\geq 0$ for $1\leq i<n$. Also, by $\sum_{i}x_{i}=0$, we have $a=\sum_{i=1}^{n-1}a_{i}=mn+k$.

The above assignment is obviously injective. 

For surjectivity, we take an $(n-1)$-tuple of nonnegative integers $(a_{1},\ldots ,a_{n-1})$. We can assume $(a_{1},\ldots ,a_{n-1})\ne (0,\ldots ,0)$. Then we can take two integers $m$ and $k$ such that $a=\sum_{i=1}^{n-1}a_{i}=mn+k$ with $1\leq k\leq n$. We assign $x\in \Lam$ by the same equation (4.2) as above, which is well-defined by $a=mn+k$. Obviously, $x$ corresponds to $(a_{1},\ldots ,a_{n-1})$ by the above assignment. This proves the bijectivity.

Note that by convention, we have $a_{n}=0$, $a_{0}=-1$, and $a_{n+i}=a_{i}+1$. Then the relation (4.1) follows directly from (4.2). 
\end{proof}

We often abbreviate $\lam ((a_{i})_{i\in \bb{Z}})$ by $\lam (a_{i})$ or $\lam (a_{1},\ldots ,a_{n-1})$ for simplicity.

For later use, we examine a stratification on affine Grassmannian by $I$-orbits using the above parametrization of $\Lam$. The key result here is the closedness of certain unions of $I$-orbits. Let us begin with recalling the following two well-known facts.

\begin{thm}[\cite{MR1923198}, Theorem 5.1.5]
The affine Grassmannian decomposes as a disjoint union $$G(F)/G(\mca{O})=\bigsqcup _{w\in W^{\mr{f}}}IwG(\mca{O})/G(\mca{O}).$$ The closure of each cell $IwG(\mca{O})/G(\mca{O})$ is a union of cells as follows: $$\overline {IwG(\mca{O})/G(\mca{O})}=\bigsqcup _{\substack{y\leq w \\ y\in W^{\mr{f}}}}IyG(\mca{O})/G(\mca{O}).$$ Here $\leq$ is the Bruhat order on $W_{\mr{aff}}$.
\end{thm}

\begin{prop}[Length formula]
Let $\lam =(\lambda_{1},\ldots ,\lambda_{n})\in \Lam \subset \bb{Z}^{n}$ be an element of $\Lam$ considered as an element of $\bb{Z}^{n}$, and $w\in \mf{S}_{n}=W$ an element of the symmetric group. Then the length of $t_{\lam}w\in W_{\mr{aff}}$ is expressed as follows: $$\ell (t_{\lam}w)=\sum _{\substack{1\leq i<j\leq n \\ w^{-1}(i)<w^{-1}(j)}}|\lambda_{j}-\lambda_{i}| +\sum _{\substack{1\leq i<j\leq n \\ w^{-1}(i)>w^{-1}(j)}}|\lambda_{j}-\lambda_{i}-1|.$$
\end{prop}

The following easy lemma describes the $W_{\mr{aff}}$-action on $\Lam$ in terms of $P$. Recall that we use cyclic notation of index for $s_{i}$ and $a=a_{1}+\cdots +a_{n-1}$.

\begin{lem}
Let $(a_{i})_{i\in \bb{Z}}$ be an element of $P$. Then the actions of simple reflections $s_{i}$ on $\lam (a_{i})$ is expressed as follows:
\begin{equation*}
s_{a+l}\cdot \lam (a_{1},\ldots ,a_{n-1})=
\begin{cases}
\lam (a_{1},\ldots ,a_{l+1},a_{l},\ldots ,a_{n-1}) & \mbox{if }l=1,2,\ldots ,n-2,\\
\lam (a_{n-1}-1,a_{1},\ldots ,a_{n-2}) & \mbox{if }l=n-1\mbox{ and }a_{n-1}\geq 1, \\
\lam (a_{1},\ldots ,a_{n-1}) & \mbox{if }l=n-1\mbox{ and }a_{n-1}=0, \\
\lam (a_{2},\ldots ,a_{n-1},a_{1}+1) & \mbox{if }l=0.
\end{cases}
\end{equation*}
\end{lem}

\begin{proof}
This can be seen by (4.2). For example, we have $$s_{a-1}\cdot \lam (a_{i})=(a_{n-k+1}-m-1,\ldots ,a_{n-2}-m-1,-m-1,a_{n-1}-m-1,a_{1}-m,\ldots ,a_{n-k}-m).$$ 

If $a_{n-1}=0$, this equals to $\lam(a_{i})$. 

If $a_{n-1}\geq 1$, then $\max\{i\in \{1,2,\ldots ,n\} \mid x_{i}=\min_{j}\{x_{j}\}\}$ decreases by one. By (4.2) where we replace $k$ by $k-1$ (in case of $k=1$, we replace $k$ by $n$ and $m$ by $m-1$), we find that $ s_{a-1}\cdot \lam (a_{1},\ldots ,a_{n-1})=\lam (a_{n-1}-1,a_{1},\ldots ,a_{n-2})$. 

The other relations are checked in the same way.
\end{proof}

In order to investigate the closure relation of $I$-orbits on affine Grassmannian, we need to understand the Bruhat order on $W^{\mr{f}}$. The next proposition gives an algorithm for calculating reduced expression of elements of $W^{\mr{f}}$.

\begin{prop}
Let $(a_{i})_{i\in \bb{Z}}$ be an nonzero element of $P$. Let $$l=\mr{max}\left\{ i\in \left\{ 1,\ldots  ,n-1\right\} \mid a_{i}\geq 1\right\}.$$ Then we have 
\begin{align*}
\ell ^{\mr{f}}\left( s_{a+l}\cdot \lam (a_{1},\ldots ,a_{n-1})\right) &=\ell ^{\mr{f}}\left( \lam (a_{1},\ldots ,a_{n-1})\right) -1,\\
\ell ^{\mr{f}}\left( s_{a}\cdot \lam (a_{1},\ldots ,a_{n-1})\right) &=\ell ^{\mr{f}}\left( \lam (a_{1},\ldots ,a_{n-1})\right) +1.
\end{align*}
\end{prop}

\begin{proof}
We first remark that by the length formula, in order to minimize $\ell (t_{\lam}w)$ (with $\lam=(\la_{1},\ldots ,\la_{n})\in \Lam$ fixed) by varying $w\in \mf{S}_{n}$, we should have 
\begin{align*}
\mbox{if $i<j$ and $\la_{j}-\la_{i}\leq 0$, then }w^{-1}(i)<w^{-1}(j), \mbox{ and}\\
\mbox{if $i<j$ and $\la_{j}-\la_{i}\geq 1$, then }w^{-1}(i)>w^{-1}(j).
\end{align*}
Hence we have 
\begin{equation}
\ell^{\mr{f}}(\lam)=\sum_{i<j}|\la_{j}-\la_{i}|-\#\{(i,j)\mid i<j, \la_{j}-\la_{i}\geq 1\}.
\end{equation}

Recall that in the notations of Proposition 4.1, we have $$\lam(a_{i})=(a_{n-k+1}-m-1,\ldots ,a_{n-1}-m-1,-m-1,a_{1}-m,\ldots ,a_{n-k}-m).$$ 

In the case of $l\ne n-k$, if we apply $s_{a+l}$ to $\lam(a_{i})$, the number of pairs $(i,j)$ with $i<j$ and $\la_{j}-\la_{i}\geq 1$ increases by one while $\sum_{i<j}|\la_{j}-\la_{i}|$ does not change. Hence we have $\ell^{\mr{f}}(s_{a+l}\cdot\lam(a_{i}))=\ell^{\mr{f}}(\lam(a_{i}))-1$. 

Therefore, we can assume $l=n-k$. Then we have 
\begin{eqnarray*}
\lam(a_{i})=&(-m-1,\ldots,-m-1,a_{1}-m,\ldots,a_{l}-m),\\
s_{a+l}\cdot \lam(a_{i})=&(a_{l}-m-1,-m-1,\ldots,-m-1,a_{1}-m,\ldots,a_{l-1}-m,-m).
\end{eqnarray*}

By (4.3), we have $$\ell^{\mr{f}}(\lam(a_{i}))=(n-l)(a_{1}+\cdots +a_{l})+\sum_{1\leq i<j\leq l}|a_{j}-a_{i}|-\#\{(i,j)\mid 1\leq i<j\leq l,a_{j}-a_{i}\geq 1\},$$ and 
\begin{align*}
\ell^{\mr{f}}(s_{a+l}\cdot\lam (a_{i}))& =(n-l-1)a_{l}+|1-a_{l}|+(a_{1}+\cdots +a_{l-1})\\
& \hspace{1em}+(n-l-1)(a_{1}+\cdots +a_{l-1})+(n-l-1)(|-m+m+1|-1)\hspace{4em}\\
& \hspace{1em}+\left(\sum_{1\leq i\leq l-1}|a_{i}-a_{l}+1|-\# \{1\leq i\leq l-1\mid a_{i}-a_{l}+1\geq 1\}\right)\hspace{3em}\\
& \hspace{1em}+\left(\sum_{1\leq i<j\leq l-1}|a_{j}-a_{i}|-\# \{ (i,j)\mid 1\leq i<j\leq l-1,a_{j}-a_{i}\geq 1\}\right)\hspace{2em}\\
& =(n-l)(a_{1}+\cdots +a_{l})-1+\sum_{1\leq i<j\leq l}|a_{j}-a_{i}|-\# \{ (i,j)\mid 1\leq i<j\leq l,a_{j}-a_{i}\geq 1\}\\
& =\ell^{\mr{f}}(\lam(a_{i}))-1.\hspace{8em}
\end{align*}
Here, in the first equality, the first part $(n-l-1)a_{l}$ is the contributions of pairs $(a_{l}-m-1,-m-1)$. The second part is the contribution of $(a_{l}-m-1,-m)$. The third part is the contributions of $(a_{i}-m,-m)$. The fourth part is the contributions of $(-m-1,a_{i}-m)$. The fifth part is the contributions of $(-m-1,-m)$. The sixth part is the contributions of $(a_{l}-m-1,a_{i}-m)$. The seventh part is the contributions of $(a_{i}-m,a_{j}-m)$. 

This proves first part of the proposition. 

Second part is proved similarly.
\end{proof}

Let $w_{\lam}$ be the element of $W=\mf{S}_{n}$ such that $\ell(t_{\lam}w_{\lam})=\ell^{\mr{f}}(\lam)$, that is, $t_{\lam}w_{\lam}\in W^{\mr{f}}$. Note that by the proof of the above proposition, the $w_{\lam}$ is the minimum of $w\in \mf{S}_{n}$ such that $w^{-1}\cdot \lam$ is antidominant. 

As a corollary of the above proposition, we get a stratification on the affine Grassmannian which will produce the desired stratification of affine Springer fibers.  

\begin{cor}
Let $c$ be a positive integer. Then the following union of $I$-orbits in the affine Grassmannian $$X_{\leq c}:=\bigsqcup _{\substack{(a_{i})\in P \\ a(a_{i})\leq c}}I\epsilon ^{\lam (a_{i})}G(\mca{O})/G(\mca{O}) \subset G(F)/G(\mca{O})$$ is closed.
\end{cor}

\begin{proof}
For each nonnegative integer $i$, let $\mr{cyc}_{i}=s_{i-1}s_{i-2}\cdots s_{1}s_{0}\in W_{\mr{aff}}$ be a cyclic product of $i$ simple reflections. 

We claim that $a=\max\{j\mid t_{\lam(a_{i})}w_{\lam(a_{i})}\geq \mr{cyc}_{j}\}$. If we assume the claim, then it becomes obvious that $t_{\lam(a_{i})}w_{\lam(a_{i})}\leq t_{\lam(a'_{i})}w_{\lam(a'_{i})}$ implies $a\leq a'=\sum_{i=1}^{n-1}a'_{i}$ for any two elements $(a_{i})$, $(a'_{i})$ of $P$. This implies the corollary.

Now we prove the claim. Let $t_{\lam(a_{i})}w_{\lam(a_{i})}=s_{i_{1}}s_{i_{2}}\cdots s_{i_{m}}$ be a reduced expression of an element of $W^{\mr{f}}$, which we label by $\mathbf{i}=(i_{1},i_{2},\ldots,i_{m})$. Let $$b=b_{\mathbf{i}}=\max\{k\in\bb{N}\mid\mr{cyc}_{k}\mbox{ is a subword of }s_{i_{1}}s_{i_{2}}\cdots s_{i_{m}}\}.$$ Note that we have $s_{i_{p}}s_{i_{p-1}}\cdots s_{i_{1}}t_{\lam(a_{i})}w_{\lam(a_{i})}\in W^{\mr{f}}$. It is enough to prove $a=b_{\mathbf{i}}$ for any reduced expression $\mathbf{i}$.

We define $r_{0},r_{1},\ldots,r_{b}$ inductively as 
\begin{align*}
r_{0}&=m+1,\\
r_{k}&=\max \{ j\mid j<r_{k-1}\mbox{ and }s_{i_{j}}=s_{k-1}\}\mbox{ for }0<k\leq b,
\end{align*}
and we set $r_{b+1}=0$.

Let $q:W_{\mr{aff}}=\Lam \rtimes W\rightarrow \Lam$ be the natural projection. Let $(a_{i}^{(p)})\in P$ be the element determined by $q(s_{i_{p}}s_{i_{p-1}}\cdots s_{i_{1}}\cdot t_{\lam(a_{i})}w_{\lam(a_{i})})=\lam(a_{i}^{(p)})$ and $a^{(p)}=a(a_{i}^{(p)})$. Note that we have $\lam(a_{i}^{(p+1)})=s_{i_{p+1}}\cdot\lam(a_{i}^{(p)})$.

We prove by induction on $k$ that $a^{(r_{k+1})}=a^{(r_{k+1}+1)}=\cdots=a^{(r_{k}-1)}=k$. This is obvious for $k=0$. By induction hypothesis, we have $a^{(r_{k})}=k-1$. By Lemma 4.4 for $l=0$ and the equality $\lam(a_{i}^{(r_{k}-1)})=s_{i_{r_{k}}}\cdot\lam(a_{i}^{(r_{k})})=s_{k-1}\cdot\lam(a_{i}^{(r_{k})})$, we have $a^{(r_{k}-1)}=k$. For $r_{k+1}\leq p\leq r_{k}-1$, we show by descending induction on $p$ that $a^{(p)}=k$. This is already proved for $p=r_{k}-1$. By the choice of $r_{k+1}$, we have $s_{i_{p+1}}\neq s_{k}$. We claim that $s_{i_{p+1}}\neq s_{k-1}$. Indeed, assume that we have $s_{i_{p+1}}=s_{k-1}$. If $a_{n-1}^{(p+1)}\geq 1$, then we have by $a^{(p+1)}=k$ and Proposition 4.5 (in this case, we have $l=n-1$ in the notation of Proposition 4.5), $\ell^{\mr{f}}(s_{i_{p+1}}\cdot\lam(a_{i}^{(p+1)}))=\ell^{\mr{f}}(\lam(a_{i}^{(p+1)}))-1$. If $a_{n-1}^{(p+1)}=0$, then we have $s_{i_{p+1}}\cdot\lam(a_{i}^{(p+1)})=\lam(a_{i}^{(p+1)})$. In both cases, these equality contradict to the equality
\begin{align*}
\ell^{\mr{f}}(s_{i_{p+1}}\cdot\lam(a_{i}^{(p+1)}))&=\ell^{\mr{f}}(\lam(a_{i}^{(p)}))\\
                                                  &=\ell (s_{i_{p}}s_{i_{p-1}}\cdots s_{i_{1}}\cdot t_{\lam (a_{i})}w_{\lam (a_{i})})\\
                                                  &=\ell (s_{i_{p+1}}s_{i_{p}}\cdots s_{i_{1}}\cdot t_{\lam (a_{i})}w_{\lam (a_{i})})+1\\
                                                  &=\ell^{\mr{f}}(\lam(a_{i}^{(p+1)}))+1.
\end{align*}
Therefore, we have $s_{i_{p+1}}\neq s_{k},s_{k-1}$. Then, by looking at Lemma 4.4, we find that the action of $s_{i_{p+1}}$ does not change the value of $a$. This proves $a^{(p)}=a^{(p+1)}=k$. Hence we have $a^{(r_{k+1})}=a^{(r_{k+1}+1)}=\cdots=a^{(r_{k}-1)}=k$ for $k=0,1,\ldots, b$. In particular, we have $a=a^{(0)}=a^{(r_{b+1})}=b$.

\end{proof}

\subsection{Combinatorial descriptions of affine pavings}\mbox{}

Recall that by Theorem 2.7, intersections of $I$-orbits on the affine flag variety and the affine Springer fiber are affine spaces if they are nonempty. In this section, we give a combinatorial criterion for which intersections are nonempty and a combinatorial formula for their dimensions. We set $C_{\lam(a_{i})}=I\epsilon^{\lam(a_{i})}G(\mca{O})/G(\mca{O})\cap X_{v}$.

\subsubsection*{The case of the affine Grassmanian}\mbox{}

Following proposition determines which $I$-orbit intersects with $X_{v}$. 

\begin{prop}
Let $(a_{i})$ be an element of $P$. Then the following two conditions are equivalent:
\begin{enumerate}
\item $C_{\lam(a_{i})}\neq\emptyset$,
\item for any $k\in \bb{Z}$, $a_{k}-a_{k+b}\leq m$.
\end{enumerate}
\end{prop}

\begin{proof}
We have a $\bb{G}_{m}$-action on $X_{v}$ with its fixed points contained in a discrete subset $\{ \epsilon^{\lam}G(\mca{O})/G(\mca{O})\mid \lam\in\Lam\}$ of $X$. Hence $C_{\lam(a_{i})}$ is not empty if and only if $\epsilon ^{\lam (a_{i})}G(\mca{O})/G(\mca{O})$ is contained in $X_{v}$. 

This is equivalent to 
\begin{equation}
\mr{Ad}(\epsilon ^{\lam (a_{i})})^{-1}v=\sum_{\alpha\in\Delta_{b-n}}\epsilon^{m+1-\langle \lam(a_{i}),\alpha\rangle}e_{\alpha}+\sum_{\alpha\in\Delta_{b}}\epsilon^{m-\langle\lam(a_{i}),\alpha\rangle}e_{\alpha}\in \mf{g}[[\epsilon]].
\end{equation}

Recall that $\alpha_{n}=-\alpha_{1}-\cdots-\alpha_{n-1}$ and $\alpha_{i+n}=\alpha_{i}$ by our convention. Then any $\alpha\in\Delta_{b}$ can be written as a sum of $b$ simple roots $\alpha=\alpha_{k}+\alpha_{k+1}+\cdots+\alpha_{k+b-1}$ with $k\in\bb{Z}$ such that $\{k,k+1,\ldots,k+b-1\}\cap n\bb{Z}=\emptyset$. Also, any $\alpha\in\Delta_{b-n}$ can be written as $\alpha=\alpha_{k}+\alpha_{k+1}+\cdots+\alpha_{k+b-1}$ with $k\in\bb{Z}$ such that $\{k,k+1,\ldots,k+b-1\}\cap n\bb{Z}\ne\emptyset$. 

By Proposition 4.1, we have
\begin{equation*}
\begin{cases}
\langle\lam(a_{i}),\alpha_{k}+\cdots+\alpha_{k+b-1}\rangle=a_{k-a}-a_{k+b-a} & \mbox{if }\{k,k+1,\ldots,k+b-1\}\cap n\bb{Z}=\emptyset,\\
\langle\lam(a_{i}),\alpha_{k}+\cdots+\alpha_{k+b-1}\rangle=a_{k-a}-a_{k+b-a}+1 & \mbox{if }\{k,k+1,\ldots,k+b-1\}\cap n\bb{Z}\ne\emptyset.
\end{cases}
\end{equation*}

Hence (4.4) is equivalent to
\begin{equation*}
\begin{cases}
a_{k-a}-a_{k+b-a}\leq m & \mbox{for any $k$ such that }\{k,k+1,\ldots,k+b-1\}\cap n\bb{Z}=\emptyset,\\
a_{k-a}-a_{k+b-a}+1\leq m+1 & \mbox{for any $k$ such that }\{k,k+1,\ldots,k+b-1\}\cap n\bb{Z}\ne\emptyset.
\end{cases}
\end{equation*}
This proves the proposition.
\end{proof}

\begin{prop}
We have a bijection $$\left\{(a_{i})\in P\mid C_{\lam(a_{i})}\ne\emptyset\right\}\longrightarrow\left\{\la\mid\la\mbox{ is a partition},\la\subseteq\delta-\delta'\right\},$$ given by $$(a_{i})\longmapsto\la(a_{i}):=\delta-(a_{b},a_{2b},\ldots,a_{nb}).$$
\end{prop}

\begin{proof}
By Proposition 4.7, the LHS of the assertion is equal to $$\{(a_{i})_{i\in\bb{Z}}\in P\mid a_{k}-a_{k+b}\leq m \mbox{ for any }k\in\bb{Z}\}.$$ 

On the other hand, for $\delta-(a_{b},a_{2b},\ldots,a_{nb})$ to be a partition, we need $$(n-1)m+b-1-a_{b}\geq (n-2)m+b-1-a_{2b}\geq\cdots\geq m+b-1-a_{(n-1)b}\geq b-1-a_{nb}\geq 0.$$ 

This is equivalent to 
\begin{equation*}
\begin{cases}
a_{lb}-a_{(l+1)b}\leq m & \mbox{for any }l\in\{1,2,\ldots,n-1\},\\
a_{nb}\leq b-1.
\end{cases}
\end{equation*}
Since $a_{nb}=a_{n}+b-1$, the second condition above is equivalent to $a_{n}\leq 0$.

Recall that $\delta'=(\delta'_{1},\ldots,\delta'_{n})$ is given by $\delta'_{l}=l'$, where $lb=l'n+l''$ with $1\leq l''\leq n$. Hence the condition $\delta-(a_{b},\ldots,a_{nb})\subseteq \delta-\delta'$ is equivalent to $a_{lb}\geq \delta'_{l}$, which is also equivalent to $a_{l''}\geq 0$ by $a_{lb}=a_{l'n+l''}=l'+a_{l''}$. 

Since $b$ is prime to $n$, the last condition is equivalent to $a_{i}\geq 0$ for any $i\in\{1,\ldots,n\}$.

Therefore, we have $a_{n}=0$, which implies $(a_{i})\in P$, and we have $a_{0}-a_{b}=-1-a_{b}\leq -1\leq m$. This implies that $a_{lb}-a_{(l+1)b}\leq m$ for any $l\in\{1,\ldots,n\}$. Again by $(n,b)=1$, this is equivalent to $a_{k}-a_{k+b}\leq m$ for any $k\in\bb{Z}$. Hence the assertion follows.
\end{proof}

Since $(n,b)=1$, the set $\{ \pm (\varepsilon_{a+ib}-\varepsilon_{a+i'b})\mid 0\leq i<i'\leq n-1\}$ is equal to the set of roots $\Delta$. We set $\la =\la(a_{i})$ where $\la(a_{i})$ is as in Proposition 4.8. We associate to an ordered pair of elements of $\la +(1^{n})/\la$ an element of $\Delta$ in the below. 

Let $x=(i,j)$ and $y=(i',j')$ be two distinct elements in $\la +(1^{n})/\la$ . Note that by the definition of $\la(a_{i})$, we have $j=m(n-i-1)+b-1-a_{(i+1)b}$ and similar formula for $j'$. Therefore, we have $$d_{m}(x)-d_{m}(y)=a_{(i'+1)b}-a_{(i+1)b}.$$ 

We set $$\alpha(x,y)=\epsilon_{a+(i+1)b}-\epsilon_{a+(i'+1)b}.$$ Then we have $\alpha(y,x)=-(\epsilon_{a+(i+1)b}-\epsilon_{a+(i'+1)b})$. 

This correspondence is used in the proof of Proposition 4.10.

\begin{lem}
Let $x,y$ be as above. Then we have $$\langle \lam(a_{i}),\alpha(x,y)\rangle =\begin{cases} d(x,y) & \mbox{if }\alpha(x,y)\in \Delta^{+}, \\ d(x,y)+1 & \mbox{if }\alpha(x,y)\in \Delta^{-}. \end{cases} $$
\end{lem}

\begin{proof}
Let $l'$ be the integer satisfying  $(i'-i)b=l'n+l''$, where $1\leq l''\leq n-1$.
If $i<i'$, the assertion follows immediately from Proposition 4.1 by 
\begin{equation*}
\#\left(\{a+(i+1)b,a+(i+1)b+1,\ldots,a+(i'+1)b-1\}\cap n\bb{Z}\right)=
\begin{cases}l' & \mbox{if }\alpha(x,y)\in \Delta^{+},\\
l'+1 & \mbox{if }\alpha(x,y)\in \Delta^{-}.\\
\end{cases}
\end{equation*}
If $i>i'$, then we have 
\begin{eqnarray*}
\langle \lam(a_{i}),\alpha(x,y)\rangle &=&-\langle \lam(a_{i}),\alpha(y,x)\rangle \\&=& \begin{cases} -(d(y,x)+1) & \mbox{if }\alpha(y,x)\in \Delta^{-} \\ -d(y,x) & \mbox{if }\alpha(y,x)\in \Delta^{+} \end{cases}\\ &=&\begin{cases} d(x,y) & \mbox{if }\alpha(x,y)\in \Delta^{+} \\ d(x,y)+1 & \mbox{if }\alpha(x,y)\in \Delta^{-}. \end{cases}
\end{eqnarray*}
\end{proof}

Next proposition determines the dimensions of affine cells in terms of combinatorial data. Let $(a_{i})$ be an element of $P$ satisfying the conditions of Proposition 4.8, and $\la=\la(a_{i})\subset\delta-\delta'$. We set $$d(\la)=\dim C_{\lam(a_{i})}.$$ 
\begin{prop}
We have $$d(\la)+m(\la)=m\binom {n}{2} +\frac{(n-1)(b-1)}{2}.$$
\end{prop}

\begin{proof}
Let $x=(i,j)$, $y=(i',j')$ be two elements of $\la+(1^{n})/\la$ with $x>_{d}y$.  

By Theorem 2.7, we have $$d(\la)=\#\{(\alpha,k)\in\Delta\times\bb{Z}\mid 0<\langle\check{\rho},\alpha\rangle+kn<mn+b,\mbox{ and }\langle-\lam(a_{i}),\alpha\rangle+k<0\}.$$ Using this formula, we can determine the contribution of $\alpha(x,y)$ to $d(\la)$. 

We first assume that $\alpha(x,y)\in\Delta^{+}$. Then we have $\langle\check{\rho},\alpha(x,y)\rangle=l(x,y)$ and $\langle\lam(a_{i}),\alpha(x,y)\rangle=d(x,y)$. 

For $k\in\bb{Z}$ to satisfy $0<\langle\check{\rho},\alpha(x,y)\rangle+kn<mn+b$, we need 
\begin{equation*}
\begin{cases}
0\leq k\leq m \mbox{ if } l(x,y)<b, \\
0\leq k\leq m-1 \mbox{ if } l(x,y)\geq b.
\end{cases}
\end{equation*}

Hence the number of $k$'s which satisfies the above two conditions for $\alpha(x,y)$ is 
\begin{equation*}
\begin{cases}
\mr{min}(m+1,d(x,y)) & \mbox{ if }l(x,y)<b,\\
\mr{min}(m,d(x,y)) & \mbox{ if }l(x,y)\geq b.
\end{cases}
\end{equation*}
For $-\alpha(x,y)$, there is no $k$'s satisfying the above two conditions. 

The case of $\alpha(x,y)\in\Delta^{-}$ can be treated similarly. In this case, we have $\langle\check{\rho},\alpha(x,y)\rangle=l(x,y)-n$ and $\langle\lam(a_{i}),\alpha(x,y)\rangle=d(x,y)+1$. 

For $k\in\bb{Z}$ to satisfy $0<\langle\check{\rho},\alpha(x,y)\rangle+kn<mn+b$, we need 
\begin{equation*}
\begin{cases}
1\leq k\leq m+1 \mbox{ if } l(x,y)-n<b-n, \\
1\leq k\leq m \mbox{ if } l(x,y)-n\geq b-n.
\end{cases}
\end{equation*}

Hence the number of $k$'s satisfying the above two conditions is the same as in the case of $\alpha(x,y)\in\Delta^{+}$.

Next we consider the contribution of $(x,y)$ to $m(\la)$. By the definition of $\mr{dinv}(T)$, for $T\in\mr{SSYT}(\la+(1^{n})/\la)$ to attain the maximum, it is necessary that $T(x)<T(y)$. Then the contribution of $(x,y)$ is $m(x,y)$ or $m(x,y)\pm 1$. Note that the contribution of $(x,y)$ to $d(\la)+m(x,y)$ is $m+1$ if $l(x,y)<b$ and $m$ if $l(x,y)\geq b$. 

We set $$C=\{(x,y)\in\left(\la+(1^{n})/\la\right)\times\left(\la+(1^{n})/\la\right)\mid x>_{d}y\mbox{ and }l(x,y)<b\}.$$ 

Then by the definition of $\mr{dinv}$, we have $$d(\la)+m(\la)=m\binom {n}{2}+\#C-\#A+\#B,$$ where $A$ and $B$ are borrowed from section 3.2.

We set 
\begin{eqnarray*}
D\hspace{-2.5em}& =\{(x,y)\in\left(\la+(1^{n})/\la\right)\times\left(\la+(1^{n})/\la\right)\mid x>_{d}y,i<i'\mbox{ and }l(x,y)<b\}\\ 
& \hspace{4em}\bigcup\{(x,y)\in\left(\la+(1^{n})/\la\right)\times\left(\la+(1^{n})/\la\right)\mid x>_{d}y,i>i'\mbox{ and }l(y,x)<b\}.
\end{eqnarray*}

By (3.1) and (3.2), we have $C\setminus (C\cap D)=A$ and $D\setminus (C\cap D)=B$. Hence we have $\#C-\#D=\#A-\#B$. Therefore, it suffices to prove that $\#D=\frac{(n-1)(b-1)}{2}$. 

We set $$E=\left\{k\in\{2,\ldots,n-1\}\mid kb\equiv 1,2,\ldots,b-1\\ \mod n\right\}.$$ We have $\#E=b-1$. Note that $k\in E$ if and only if $n+1-k\in E$. 

Then we have 
\begin{eqnarray*}
\#D&=&\#\{(i,i')\in \{0,1,\ldots,n-1\}\times\{0,1,\ldots,n-1\}\mid i'-i\in E\}\\
&=&\sum_{k\in E}(n-k)\\
&=&n(b-1)-\frac{1}{2}\sum_{k\in E}\{k+(n+1-k)\}\\
&=&n(b-1)-\frac{(n+1)(b-1)}{2}=\frac{(n-1)(b-1)}{2},
\end{eqnarray*}
as required.
\end{proof}

\subsubsection*{The parahoric cases}\mbox{}

Let $\mu =(\mu_{1},\mu_{2},\ldots ,\mu_{l})$ be a partition of $n$. Let $\Xi_{i}=\{\tilde{\mu}_{i-1}+1,\tilde{\mu}_{i-1}+2,\ldots,\tilde{\mu}_{i}\}\subset\{1,2,\ldots,n\}$. We define $r_{\mu}:\{1,2,\ldots ,n\}\rightarrow \{\bar{1},\bar{2},\ldots ,\bar{l}\}\subset \mca{A}_{-}$ by $r_{\mu}(\Xi_{i})=\{\bar{i}\}$.

For $l\in \bb{Z}$, we write $\underline{l}$ for the integer such that $l=l'n+\underline{l}$, with $1\leq \underline{l}\leq n$.

\begin{prop}
We keep the setting of Proposition 4.7. We have a bijection $$\left\{ w\in W/W_{\mu}\mid I\epsilon^{\lam(a_{i})}w\hat{P}_{\mu}/\hat{P}_{\mu}\cap \hat{\mca{P}}_{\mu,v}\ne \emptyset \right\}\longrightarrow \mr{SSYT}_{-}(\la +(1^{n})/\la, \mu),$$ given by $$w\longmapsto \{ (i,j)\mapsto r_{\mu}(w^{-1}(\underline{a+(i+1)b}))\}.$$
\end{prop}

\begin{proof}
Let $w\in W$ be a representative of an element of the LHS of the assertion. Then we have $$w^{-1}\left(\mr{Ad}(\epsilon^{\lam(a_{i})})^{-1}v\right)\in\mr{Lie}\hat{P}_{\mu}.$$ 

Since we have $\mr{Ad}(\epsilon^{\lam(a_{i})})^{-1}v\in \mf{g}[[\epsilon]]$, this is equivalent to $\overline{w^{-1}\left(\mr{Ad}(\epsilon^{\lam(a_{i})})^{-1}v\right)}\in\mf{p}_{\mu}$. Here, for each $u\in\mf{g}[[\epsilon]]$, we denote by $\bar{u}$ the image under the evaluation map $\mf{g}[[\epsilon]]\rightarrow\mf{g}$, $\epsilon\mapsto 0$. 

By the proof of Proposition 4.7, for $e_{\alpha}$ ($\alpha=\varepsilon_{k}-\varepsilon_{k+b}\in\Delta_{b}\sqcup\Delta_{b-n}$) to appear in $\overline{\left(\mr{Ad}(\epsilon^{\lam(a_{i})})^{-1}v\right)}$ with nonzero coefficient, it is necessary and sufficient that $a_{k-a}-a_{k+b-a}=m$. Let $k=a+(i+1)b$, $i=0,1.\ldots,n-1$. Then the above condition equals to $a_{(i+1)b}-a_{(i+2)b}=m$, and any element of $\Delta_{b}\sqcup\Delta_{b-n}$ can be written as $\varepsilon_{k}-\varepsilon_{k+b}$ since $n$ and $b$ are coprime. Let $x=(i,j)$, $y=(i+1,j')\in\la+(1^{n})/\la$. Then the condition $a_{(i+1)b}-a_{(i+2)b}=m$ is equivalent to $j=j'$ by Proposition 4.8.

We have $e_{w^{-1}(\alpha)}\in\mf{p}_{\mu}$ if and only if $$r_{\mu}(w^{-1}(\underline{a+(i+1)b}))\leq r_{\mu}(w^{-1}(\underline{a+(i+2)b})).$$ Hence $\overline{w^{-1}\left(\mr{Ad}(\epsilon^{\lam(a_{i})})^{-1}v\right)}\in\mf{p}_{\mu}$ is equivalent to the condition that for any $i$ with $a_{(i+1)b}-a_{(i+2)b}=m$, we have $r_{\mu}(w^{-1}(\underline{a+(i+1)b}))\leq r_{\mu}(w^{-1}(\underline{a+(i+2)b}))$. This means that the tableau on $\la+(1^{n})/\la$ as in the statement of the proposition is semistandard. Here, we note that we use negative letters so that semistandard tableau can have equal entries on the same column. 

This proves that the map in the statement of the assertion is well-defined. Surjectivity follows by reversing the above construction, and injectivity follows from the observation that $\mf{S}_{\mu}$ acts transitively on each $\Xi_{i}$.
\end{proof}

\begin{prop}
Suppose that $w\in W/W_{\mu}$ and $T\in\mr{SSTY}_{-}(\la +(1^{n})/\la,\mu)$ correspond to each other under the bijection of Proposition 4.11. Then we have $$\dim (I\epsilon^{\lam(a_{i})}w\hat{P}_{\mu}/\hat{P}_{\mu}\cap\hat{\mca{P}}_{\mu,v})=d(\la)+\mr{dinv}''(T).$$
\end{prop}

\begin{proof}
By Theorem 2.7 applied for $y=c\check{\rho_{\mu}}$, with $0<c\ll 1$, the LHS is 
\begin{equation}
\#\left\{(\alpha,k)\in\Delta\times\bb{Z}\mid 0\leq\langle\check{\rho},\alpha\rangle+kn<mn+b,\langle-\lam(a_{i})+cw(\check{\rho_{\mu}}),\alpha\rangle+k<0\right\}.
\end{equation}

Note that the number of $(\alpha,k)$'s satisfying the conditions in (4.5) with second condition replaced by $\langle-\lam(a_{i}),\alpha\rangle+k<0$ is $d(\la)$. Hence it suffices to count $(\alpha,k)$'s which also satisfies $k=\langle\lam(a_{i}),\alpha\rangle$.

Let $x=(i,j)$, $y=(i',j')\in\la+(1^{n})/\la$ with $x>_{d}y$. We consider the contribution of $\pm\alpha(x,y)$. If $\alpha(x,y)\in\Delta^{+}$, we have $\langle\lam(a_{i}),\alpha(x,y)\rangle=d(x,y)\geq 0$, hence $-\alpha(x,y)\in\Delta^{-}$ does not contribute since for $k$ to satisfy the first condition in (4.5), we need $k\geq 1$. The first condition in (4.5) for $\alpha(x,y)$ means that 
\begin{equation*}
\begin{cases}
0\leq d(x,y)\leq m & \mbox{ if }l(x,y)<b,\\   
0\leq d(x,y)\leq m-1 & \mbox{ if }l(x,y)\geq b.
\end{cases}
\end{equation*}

In the case of $\alpha(x,y)\in\Delta^{-}$, we proceed as above and obtain the same condition.

Second condition in (4.5) means that $\langle\check{\rho_{\mu}},w^{-1}(\alpha(x,y))\rangle<0$. This is equivalent to $$r_{\mu}(w^{-1}(\underline{a+(i+1)b}))>r_{\mu}(w^{-1}(\underline{a+(i'+1)b})).$$
This means that $T(x)>T(y)$. Hence $\alpha(x,y)$ contributes to the LHS if and only if $(x,y)$ contributes to $\mr{dinv}''(T)$. This proves the proposition. 
\end{proof}

\subsubsection*{Proof of the main theorem}\mbox{}

Recall that $C_{\lam(a_{i})}$'s are affine spaces and form an affine paving of $X_{v}$ by Theorem 2.7. Let $\pi :\hat{\mca{B}}_{v} \rightarrow X_{v}$ and $\pi_{\mu}:\hat{\mca{P}}_{\mu,v} \rightarrow X_{v}$ be natural projections. 

\begin{thm}
Let notation be as above. Then we have $$\mca{F}(H^{\mr{BM}}_{\ast}(\pi^{-1}(C_{\lam(a_{i})})),z;q)=q^{m\binom{n}{2}+\frac{(n-1)(b-1)}{2}}\omega D^{\la}(z;q^{-1}).$$
\end{thm}

\begin{proof}
It suffices to check that the both sides paired with $h_{\mu}$ yield the same number. Indeed,
\begin{eqnarray*}
\langle\mca{F}(H^{\mr{BM}}_{\ast}(\pi^{-1}(C_{\lam(a_{i})})),z;q),h_{\mu}\rangle
&=&\sum_{i=0}^{\infty}q^{i}\dim H^{\mr{BM}}_{2i}(\pi_{\mu}^{-1}(C_{\lam (a_{i})}))\\
&=&\sum_{T\in\mr{SSYT}_{-}(\la+(1^{n})/\la,\mu)}q^{d(\la)+\mr{dinv}''(T)}\\
&=&\sum_{T\in\mr{SSYT}_{-}(\la+(1^{n})/\la,\mu)}q^{d(\la)+m(\la)-\mr{dinv}(T)}\\
&=&q^{m\binom{n}{2}+\frac{(n-1)(b-1)}{2}}\sum_{T\in\mr{SSYT}_{-}(\la+(1^{n})/\la,\mu)}q^{-\mr{dinv}(T)}\\
&=&q^{m\binom{n}{2}+\frac{(n-1)(b-1)}{2}}\langle\omega D^{\la}(z,q^{-1}),h_{\mu}\rangle.
\end{eqnarray*}
Here the first equality follows from Lemma 3.3 and Proposition 2.6. The second equality follows from Proposition 4.11 and Proposition 4.12. The third equality follows from Corollary 3.6. The fourth equality follows from Proposition 4.10. The fifth equality follows from Lemma 3.9 and $\langle m_{\la}, h_{\mu}\rangle=\delta_{\la\mu}$.
\end{proof}

\begin{cor}
For every $\la\subset(\delta-\delta')$, the symmetric function $D^{\la}(z;q)$ is Schur positive.
\end{cor}

\begin{thm}
There exists a filtration on $H^{\mr{BM}}_{\ast}(\hat{\mca{B}}_{v})$ compatible with the homological grading and the $\mf{S}_{n}$-action such that the following equation holds: $$\mca{F}(\mr{gr}_{\ast}H^{\mr{BM}}_{\ast}(\hat{\mca{B}}_{v}),z;q,t)=q^{m\binom{n}{2}+\frac{(n-1)(b-1)}{2}}\omega D(z;q^{-1},t).$$
\end{thm}

\begin{proof}
Let $Y_{\leq c}=\pi^{-1}(X_{\leq c})\cap\hat{\mca{B}}_{v}$. This is closed subset of $\hat{\mca{B}}_{v}$ by Corollary 4.6. By definition, $Y_{\leq c}$ has an affine paving which is a part of the affine paving of $\hat{\mca{B}}_{v}$. Hence we have an inclusion of Borel-Moore homology $H^{\mr{BM}}_{\ast}(Y_{\leq c})\subseteq H^{\mr{B}}_{\ast}(\hat{\mca{B}}_{v})$. 

We define an increasing filtration on $H^{\mr{BM}}_{\ast}(\hat{\mca{B}}_{v})$ by $F_{\leq c}H^{\mr{BM}}_{\ast}(\hat{\mca{B}}_{v}):=H^{\mr{BM}}_{\ast}(Y_{\leq c})$. Then the associated graded of degree $c$ is isomorphic as $\mf{S}_{n}$-module to the sum of $H^{\mr{BM}}_{\ast}(\pi^{-1}(C_{\lam(a_{i})}))$'s with $a(a_{i})=c$. Hence the theorem follows from Theorem 4.13 and Definition 3.7.
\end{proof}

\begin{rem}
In the above proof of Theorem 4.13, we used the fact that $D^{\la}(z;q)$ is a symmetric function. However, for a composition $\mu=(\mu_{i})$ of $n$, the dimensions of $\mf{S}_{\mu}$-invariants do not depend on the order of $\mu_{i}$'s. This fact and the same calculation as above implies that $D^{\la}(z;q)$ is a symmetric function.
\end{rem}

\nocite{*}
\bibliographystyle{plain}
\bibliography{Affine_Springer_Fibers_of_type_A_and_Combinatorics_of_Diagonal_Coinvariants}

\end{document}